\documentclass[referee]{svjour3}                     
\usepackage{amsfonts}
\usepackage{amssymb}
\usepackage{amsmath}
\usepackage{graphicx}
\usepackage{empheq}
\usepackage{indentfirst}
\usepackage{cite}
\usepackage{mathrsfs}
\usepackage{graphics}
\usepackage[percent]{overpic}
\usepackage{color}

\graphicspath{ {./figs_Num_CHSD/} }

\let\oldeq\equation{}\def\equation{\par\vspace{-\parskip}\oldeq}

\allowdisplaybreaks[4]

\def\be{\begin{equation}}
\def\ee{\end{equation}}
\def\bea{\begin{eqnarray}}
\def\eea{\end{eqnarray}}
\def\bt{\begin{theorem}}
\def\et{\end{theorem}}
\def\bl{\begin{lemma}}
\def\el{\end{lemma}}
\def\br{\begin{remark}}
\def\er{\end{remark}}
\def\bp{\begin{proposition}}
\def\ep{\end{proposition}}
\def\bc{\begin{corollary}}
\def\ec{\end{corollary}}
\def\bd{\begin{definition}}
\def\ed{\end{definition}}

\def\vp{\varphi}

\def\non{\nonumber }
\def\ub{\mathbf{u}}
\def\baru{\overline{\mathbf{u}}}
\def\btau{\boldsymbol{\tau}}
\def\BH{\mathbf{H}}

\def\ar1{\varphi_h^{k+1}}
\def\var0{\varphi_h^{k}}
\def\bu1{\overline{\mathbf{u}}_h^{k+1}}

\def\E{\mathbf{E}}

\newcommand{\ignore}[1]{}

\begin{document}

\title{Uniquely solvable and energy stable decoupled numerical schemes for the  Cahn-Hilliard-Stokes-Darcy system for two-phase flows in karstic geometry\thanks{Supported in part by an NSF grant {DMS-1312701}.
Wenbin Chen is supported by the NSFC (1119099, 9113004), and a 111 project (B08018). }}

\author{Wenbin Chen         \and
        Daozhi Han \and Xiaoming Wang 
}

\titlerunning{Decoupled Energy-stable Numerical Schemes for CHSD}
\authorrunning{W. Chen, D. Han and X. Wang} 

\institute{W. Chen  \at
              School of Mathematical Sciences, Fudan University,  Shanghai 200433, China \\
              \email{wbchen@fudan.edu.cn}           
           \and
           D. Han \at
              Department of Mathematics, Indiana University, Bloomington, IN 47405, U.S.A.\\
               \email{djhan@iu.edu}
            \and
            X. Wang \at
              Department of Mathematics, Florida State University, Tallahassee, FL 32306-4510, U.S.A. \\
              \email{wxm@math.fsu.edu} 
}

\date{Received: date / Accepted: date}

\maketitle

\begin{abstract}
We propose and analyze two novel decoupled numerical schemes for solving the Cahn-Hilliard-Stokes-Darcy (CHSD) model for two-phase flows in karstic geometry. In the first numerical scheme, we explore a fractional step method (operator splitting) to decouple the phase-field (Cahn-Hilliard equation) from the velocity field (Stokes-Darcy fluid equations). To further decouple the Stokes-Darcy system, we introduce a first order pressure stabilization term in the Darcy solver in the second numerical scheme so that the Stokes system is decoupled from the Darcy system and hence the CHSD system can be solved in a fully decoupled manner. We show that both decoupled numerical schemes are uniquely solvable, energy  stable, and mass conservative. Ample numerical results are presented to demonstrate the accuracy and efficiency of our schemes.

\end{abstract}



\section{Introduction}
\setcounter{equation}{0}

Many natural and engineering applications involve multiphase flows in karstic geometry, i.e.,  geometry with both conduit (or vug) and porous media \cite{HSW2014}. Such kind of problems are intrinsically difficulty due to the multi-scale multi-physics nature.  In \cite{HSW2014}, the authors utilized Onsager's extremum principle to derive a diffuse interface model, the so-called Cahn-Hilliard-Stokes-Darcy system (CHSD), for two-phase incompressible  flows with matched densities in the karstic geometry. Existence and weak-strong uniqueness of  weak solutions for the CHSD system have been proved recently in \cite{HWW2014}.  
For complex systems like the CHSD model, efficient and accurate numerical schemes are highly desirable. There are several challenges associated with the system. First,  due to the relative slow motion of fluid in porous media, long time simulations are needed in order to capture physically important phenomena. In particular, we would like to have numerical schemes that inherit, with modification if needed, the energy law of the continuous model. Second, the CHSD model, similar to all phase field model with a sharp interface limit,  is stiff due to the existence of relatively steep transition regions. This stiffness leads to a severe time-step restriction if one adopts classical explicit time stepping. Third, the CHSD system involves at least three coupled physical processes: the dynamics of the phase field variable (governed by the Cahn-Hilliard equation), the fluid flow in the conduit (governed by the Stokes equation), and the fluid flow in the porous media (governed by the Darcy system). Efficient and accurate numerical schemes for each of the sub-models do exist. Therefore, it would be highly desirable to have numerical schemes that decouple these subsystems while maintaining the long time stability. Such decoupled schemes would reduce the complexity of the computation and allow the possibility of the utilization of legacy codes.
In this paper, we introduce and analyze two novel decoupled schemes for the CHSD system. In particular, we show that both schemes are uniquely solvable and enjoy appropriate discrete energy law which ensure the long time stability.
So far as we know, these are the first set of uniquely solvable and energy stable decoupled schemes for computing two-phase flows in karstic geometry.

To fix the notation, let us assume that
the two-phase flows are confined in a bounded connected domain $\Omega \subset \mathbb{R}^d$ ($d=2,3$) with sufficiently smooth boundary. The unit outer normal at $\partial \Omega$ is denoted by $\mathbf{n}$. The domain $\Omega$ is partitioned into two non-overlapping regions such that $\overline{\Omega}=\overline{\Omega}_c\cup \overline{\Omega}_m$ and $\Omega_c \cap \Omega_m= \emptyset$, where $\Omega_c$ and $\Omega_m$ represent the underground
conduit (or vug) and the porous matrix region, respectively.
We denote $\partial \Omega_c$ and $\partial \Omega_m$ the boundaries of the conduit and the matrix part, respectively. Both $\partial \Omega_c$ and  $\partial \Omega_m$ are assumed to be Lipschitz continuous. The interface between the two parts (i.e., $\partial \Omega_c\cap \partial \Omega_m$) is denoted by $\Gamma_{cm}$, on which $\mathbf{n}_{cm}$ denotes the
unit normal to $\Gamma_{cm}$ pointing from the conduit part to the matrix part. Then we denote $\Gamma_c=\partial \Omega_c\backslash \Gamma_{cm}$ and $\Gamma_m=\partial \Omega_m\backslash \Gamma_{cm}$ with   $\mathbf{n}_c, \mathbf{n}_m$ being the unit outer normals to
$\Gamma_{c}$ and $\Gamma_{m}$.
On the conduit/matrix interface $\Gamma_{cm}$, we denote by $\{\btau_i\}$ $(i=1,...,d-1)$ a local orthonormal
basis for the tangent plane to $\Gamma_{cm}$.  A two dimensional geometry is illustrated in Figure \ref{domain}.
\begin{figure}[ht]
  \begin{center}
    \includegraphics[width=0.5\textwidth]{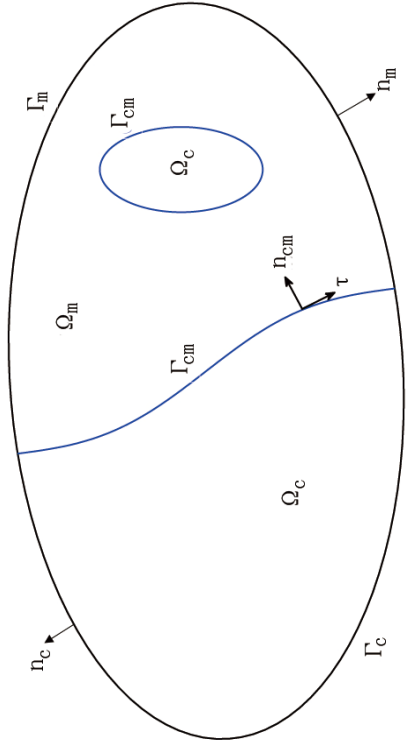} 
  \end{center}
  \caption{Schematic illustration of the domain in 2D}
  \label{domain}
\end{figure}
\\
\noindent In the sequel, the subscript $m$ (or $c$) emphasizes that the variables are for the matrix part (or the conduit part). We denote by $\ub$ the mean velocity of the fluid mixture and $\vp$ the phase function related to the concentration of the fluid (volume fraction). The following convention will be assumed throughout the paper
$$
\ub|_{\Omega_m}=\ub_m, \ \ \ub|_{\Omega_c}=\ub_c,\ \ \varphi|_{\Omega_m}=\varphi_m, \ \ \varphi|_{\Omega_c}=\varphi_c.
$$

\textbf{Governing PDE system}. We shall consider the following generalized Cahn-Hilliard-Stokes-Darcy system with time derivatives retained in the Stokes-Darcy system for generality:
\begin{eqnarray}
&&\rho_0\partial_t \ub_c=\nabla\cdot\mathbb{T}(\ub_c,P_c)-\varphi_c \nabla \mu_c,\ \ \mbox{in}\ \Omega_c,\label{HSCH-NSCH1}\\
&&\nabla\cdot\ub_c=0,\ \ \mbox{in}\ \Omega_c,\label{HSCH-NSCH2}\\
&&\partial_t\varphi_c+\nabla \cdot (\ub_c \varphi_c)={\rm div}({\rm M}(\vp_c)\nabla\mu_c),\ \ \mbox{in}\ \Omega_c,\label{HSCH-NSCH3}\\
&&\frac{\rho_0}{\chi}\partial_t \ub_m+\nu(\varphi_m)\Pi^{-1}\ub_m=-\left(\nabla P_m+ \varphi_m \nabla\mu_m\right),\ \ \mbox{in}\ \Omega_m,\label{HSCH-NSCH4}\\
&&\nabla\cdot\ub_m=0,\ \ \mbox{in}\ \Omega_m, \label{HSCH-NSCH5}\\
&&\partial_t\varphi_m+\nabla \cdot (\ub_m\varphi_m)={\rm div}({\rm M}(\vp_m)\nabla\mu_m),\ \ \mbox{in}\ \Omega_m, \label{HSCH-NSCH6}
\end{eqnarray}
where the chemical potentials $\mu_c, \mu_m$ are given by
\be
 \mu_j=\gamma[\frac{1}{\epsilon}(\vp_j^3-\vp_j)-\epsilon\Delta\varphi_j], \quad \quad j\in\{c, m\} \label{chemp}.
 \ee
The Cauchy stress tensor $\mathbb{T}$ is given by
$$\mathbb{T}(\ub_c,p_c) = 2\nu(\vp_c)\mathbb{D}(\ub_c)-P_c\mathbb{I}$$
where $\mathbb{D}(\ub_c)=\frac{1}{2}(\nabla \ub_c + \nabla^T \ub_c)$ is the  rate of strain tensor and $\mathbb{I}$ is the $d\times d$ identity matrix.
Here $\rho_0$ represents the fluid density, $\chi$ is the porosity, $\nu$ is the viscosity, $\Pi$ is the permeability matrix, the parameter $\gamma>0$ is related to the surface tension.    The  mobility of the CHSD model is denoted by  $\mathrm{M}$. Throughout, we assume that the viscosity $\nu$ and mobility $M$  are  suitable functions of the phase function $\vp$ such that $0<c\leq \nu, M\leq C$ for positive constants $c$ and $C$. We remark that the mobility function $M$ should scale like $\epsilon^2$ to recover the sharp interface model in the limit $\epsilon \rightarrow 0$, cf. \cite{MPCMC2013}.
In Eq. \eqref{HSCH-NSCH4}, $\Pi$ is a $d\times d$ matrix standing for permeability of the porous media. It is related to the hydraulic conductivity tensor of the porous medium $\mathbb{K}$ through the relation $\Pi=\frac{\nu\mathbb{K}}{\rho_0 g}$. In this manuscript,  $\mathbb{K}$ is assumed to be a bounded, symmetric and uniformly positive definite matrix.
We also adopt the convention that $\Pi$ in the denominator would be the same as multiplying the numerator by the inverse of $\Pi$ on the left.

The CHSD System   is subject to the following boundary and interface boundary conditions. \\
\noindent \textbf{Boundary conditions on $\Gamma_c$ and $\Gamma_m$}:
\begin{eqnarray}
&& \ub_c=\mathbf{0},\quad \frac{\partial\varphi_c}{\partial \mathbf{n}_c}=\frac{\partial\mu_c}{\partial \mathbf{n}_c}=0, \qquad \text{on}\ \Gamma_c,\label{IBC0}\\
&& \ub_m\cdot\mathbf{n}_m=0,\quad \frac{\partial\varphi_m}{\partial \mathbf{n}_m}=\frac{\partial\mu_m}{\partial \mathbf{n}_m}=0,\qquad \text{on}\ \Gamma_m,\label{IBC1}
\end{eqnarray}

\noindent \textbf{Interface conditions on $\Gamma_{cm}$}: 
\begin{eqnarray}
 &&\varphi_m =\varphi_c, \quad \frac{\partial \vp_m}{\partial \mathbf{n}_{cm}}=\frac{\partial \vp_c}{\partial \mathbf{n}_{cm}}, \mbox{on}\ \Gamma_{cm},\label{IBCi1}\\
&&\mu_m=\mu_c,\quad M(\varphi_m)\frac{\partial \mu_m}{\partial \mathbf{n}_{cm}}=M(\varphi_c)\frac{\partial \mu_c}{\partial \mathbf{n}_{cm}},\ \ \mbox{on}\ \Gamma_{cm}, \label{IBCi2} \\
&&\ub_m\cdot\mathbf{n}_{cm}=\ub_c\cdot\mathbf{n}_{cm},\quad \mbox{on}\ \Gamma_{cm},\label{IBCi5}\\
&&-\mathbf{n}_{cm}\cdot(\mathbb{T}(\ub_c,P_c){\mathbf{n}_{cm}})=P_m,\ \ \mbox{on}\ \Gamma_{cm},\label{IBCi6}\\
&& -\btau_i\cdot(\mathbb{T}(\ub_c,P_c){\mathbf{n}_{cm}})=\alpha_{BJSJ}\frac{\nu(\vp_m)}{\sqrt{\text{trace}(\Pi)}}\btau_i\cdot\ub_c, i=1,..,d-1,  \mbox{on}\ \Gamma_{cm}.\label{IBCi7}
\end{eqnarray} 

We refer to \cite{HSW2014} for the detailed derivation of the CHSD model \eqref{HSCH-NSCH1}--\eqref{HSCH-NSCH6} together with the interface boundary conditions \eqref{IBCi1} -\eqref{IBCi7}.
 The last interface condition \eqref{IBCi7} is
the so-called
Beavers-Joseph-Saffman-Jones (BJSJ) condition (cf. \cite{Jones1973, Saffman1971}), where $\alpha_{BJSJ}\geq 0$ is an empirical parameter assumed to be a constant here for simplicity.   The BJSJ condition is a simplified variant of the well-known Beavers-Joseph (BJ) condition (cf. \cite{BeJo1967})
that addresses the important
issue of how the porous media affects the conduit flow at the
interface:
$$-\btau_i\cdot(2\nu\mathbb
D(\mathbf{u}_c))\mathbf{n}_{cm}=\alpha_{BJ}\frac{\nu}{\sqrt{\text{trace}(\Pi)}}\btau_i\cdot(\ub_c-\ub_m), \ \ \mbox{on}\ \Gamma_{cm}, \ i=1,...,d-1.
$$ 
Mathematically rigorous justification of this simplification under appropriate assumptions can be found in \cite{JaMi2000}.

An important feature of the CHSD system \eqref{HSCH-NSCH1}--\eqref{IBCi7} is that it obeys a dissipative energy law. We define  the total energy of the coupled system as follows:
\be
\mathcal{E}(t)= \int_{\Omega_c}\frac{\rho_0}{2}|\ub_c|^2 dx+ \int_{\Omega_m}\frac{\rho_0}{2\chi}|\ub_m|^2 dx +\gamma\int_{\Omega}\left[\frac{\epsilon}{2}|\nabla\varphi|^2+\frac{1}{\epsilon}F(\varphi)\right]dx, \label{totenergy}
\ee
with $F(\vp)=\frac1{4}(\vp^2-1)^2.$
Let $(\ub_m, \ub_c, \vp)$ be a smooth solution to the initial boundary value problem \eqref{HSCH-NSCH1}--\eqref{IBCi7}. Then $(\ub_m,  \ub_c, \vp)$ satisfies the following basic energy law:
\be
\frac{d}{dt} \mathcal{E}(t)=-\mathcal{D}(t) \le 0,\quad \forall\, t\geq 0, \label{EnergyLaw}
\ee
 where the rate of energy dissipation $\mathcal{D}$ is given by
\bea
\mathcal{D}(t)&=& \int_{\Omega_m} \nu(\varphi_m)\Pi^{-1}|\ub_m|^2dx
+\int_{\Omega_c}2\nu(\vp_c)|\mathbb{D}(\ub_c)|^2dx
\non\\
&& +\int_{\Omega}{\rm M}(\vp)|\nabla\mu(\vp)|^2dx
+\alpha_{BJSJ}\int_{\Gamma_{cm}}\frac{ \nu(\vp)}{\sqrt{{\rm trace}(\Pi)}}\sum_{i=1}^{d-1}|\ub_c\cdot\btau_i|^2 dS \label{D}
\eea
where we have assumed that the Beavers-Joseph-Saffman-Jones interface parameter $\alpha_{BJSJ}$ is a constant for simplicity. The case with variable BJSJ interface parameter, a necessity for curved interface boundary, can be treated similarly.

The CHSD system  is a complicated system that couples different dynamics in different domains (Cahn-Hilliard equation, Stokes equation, Darcy equation). Hence it is of great interest to develop decoupled numerical schemes  (for instance, domain decomposition schemes) so as to  employ legacy solvers for each individual equation and reduce computational cost. On the other hand, the CHSD system is a diffuse interface model that describes physical phenomena with large gradient in a small transition layer. For such systems, unconditionally stable numerical schemes are desirable so that the stiffness can be handled with ease. It is thus crucial to design efficient decoupled stable numerical algorithms for solving this system, which is the main focus of this article.
Another challenge associated with the CHSD is the necessity for long time simulation due to the slow flow motion in the porous media vs the fast motion in the conduit for situations such as pressure gradient driven flow. Unconditional long-time stability of the schemes becomes handy although it is not equivalent to long time accuracy. (Long-time stability of the schemes is a key ingredient in ensuring the convergence of long time statistical properties for dissipative systems \cite{Wang2016}).

  Efficient solvers for each individual equation/system are building blocks for constructing an efficient numerical algorithm of the CHSD system. Among the abundant literature, we only survey those closely related to our schemes. For the Cahn-Hilliard type equation that describes physical phenomena with large gradient in a small transition layer, a popular strategy in the temporal discretization is based on a convex-splitting of the associated energy functional, see \cite{Eyre1998}  for a first order scheme and \cite{HWWL2009, SWWW2012} for second order schemes. The convex-splitting schemes are desirable because they are unconditionally energy-stable and uniquely solvable. Thus numerical stiffness can be handled with ease. There are also unconditionally stable linear schemes  in the literature \cite{ShYa2010b, GuTi2013} where additional stabilization terms are introduced to ensure stability. These ideas (convex-splitting and stabilization) have been successfully utilized in the computation of Cahn-Hilliard fluid models, cf. \cite{KKL2004, KSW2008, ShYa2010a, Grun2013, GLL2014, HaWa2015, ShYa2015, DFW2015} for Cahn-Hilliard-Navier-Stokes models and \cite{Wise2010, CSW2013} for Cahn-Hilliard-Hele-Shaw/Brinkman models. For the single-phase Stokes-Darcy system, there are many efficient numerical solvers, see for instance \cite{CR08,DMQ2002, DiQu2003, DiQu2009, LSY2002, CGHW2010, CGW2010, CGHW2011, ChGHW2011, CGSW2013,  CR09, CR12, CGR13}.

 On one hand, decoupled scheme is highly desirable for solving such a large system on a moderate computer. In addition to the apparent efficiency advantage, decoupling the computation of the system would allow the application of numerous legacy algorithms/codes surveyed above which are not directly applicable otherwise. Furthermore, the rich scales encompassed in the CHSD system naturally call for different meshes and even different time step-size for the computation of different dynamics. For instance, it is advantageous to employ adaptive mesh refinement for the computation of Cahn-Hilliard equation so as to resolve the diffuse interface of small width, especially when low order finite element is used. In contrast, the computation of fluid equation can be done on fixed coarser grids as groundwater flow is typically slow. A decoupled scheme is much easier to implement these ideas compared to coupled ones. Finally, the scheme needs to be stable for long time simulations which are typically the objective in the context of groundwater study.    

We note that a fully decoupled numerical schemes can be constructed easily, for instance, by treating the coupling terms in the equations and in the interface boundary conditions explicitly. However, such a method is not known to have the highly desirable unconditional stability for solving the CHSD system. The design of an unconditionally stable, decoupled numerical scheme requires delicate consideration and application of the classical operator-splitting/fractional-step  methodology \cite{Strang1968, Yanenko1971, Marchuk1974, Chorin1967, Temam1968a, Temam1969, Temam1969b}. The work of Temam \cite{Temam1968a, Temam1969, Temam1969b} is particularly relevant where the unconditional stability of the fractional-step schemes applied to the Navier-Stokes system was first discussed. See also \cite{KiMo1985, vanKan1986, Shen1996, GuSh2003} for 
other work related to the Chorin-Temam fractional-step method for solving fluid equations.

Another phase field model for two-phase flow in karstic geometry was proposed in \cite{CSW2014}. In their model, the Cahn-Hilliard-Navier-Stokes equation with moving contact line type boundary conditions (generalized Navier slip boundary condition for velocity and dynamic boundary condition for order parameter) is adopted for two-phase flow in conduit. In porous medium, however, a two-phase Darcy's law is utilized. The authors propose a Robin-Robin domain decomposition method to solve the coupled system. However, the stability of the proposed scheme was not discussed.
An energy stable but fully coupled scheme for the Cahn-Hilliard-Stokes-Darcy system can be found in \cite{HSW2014}.

Finally, we comment on the stability of our proposed numerical scheme. The explicit treatment of the velocity in Cahn-Hilliard equation is analyzed in \cite{KSW2008} for the Cahn-Hilliard-Navier-Stokes system. It is shown that the scheme is conditionally stable with a mild CFL condition. Extrapolation in time of the interface boundary conditions have also been used in the computation of the non-stationary Stokes-Darcy system where long-time stability and error estimates  were established under a time step-size constraint dependent on the problem parameters, see for instance \cite{MuZh2010, CGSW2013, CGSW2015}. Both decoupled schemes that we propose here are unconditionally energy stable.  Our numerical experiments verify our theoretical results on the long-time stability of the novel schemes.

The paper is structured in the following way. We introduce the function spaces and the concept of weak formulation in section 2.
Two novel decoupled  numerical schemes for solving the CHSD system are proposed and analyzed in section 3. 
In section 4, we first verify numerically that our schemes are first-order accurate in time and long-time stable. Then we present two numerical experiments, boundary driven and buoyancy driven flows,   to illustrate the effectiveness of our schemes. Both numerical simulations are of physical interest for transport processes of two-phase flow in karst geometry.


\section{The Weak formulation }\label{WM}

For our CHSD problem with domain decomposition, we introduce the following spaces
\bea
 \BH({\rm div}; \Omega_j) &:=&\{\mathbf{w}\in \mathbf{L}^2(\Omega_j)~|~\nabla \cdot \mathbf{w}\in \mathbf{L}^2(\Omega_j)\}, \quad \quad j\in \{c,m\},\non\\
\mathbf{H}_{c,0}&:=&\{\mathbf{w}\in \BH^1(\Omega_c)~|~\mathbf{w}=\mathbf{0}\text{ on
}\Gamma_{c}\},
 \non\\
\mathbf{H}_{c,\text{div}}&:=&\{
\mathbf{w}\in\mathbf{H}_{c,0}~|~\nabla \cdot\mathbf{w}=0\}, \nonumber
 \\
\mathbf{H}_{m,0}&:=&\{\mathbf{w}\in \BH({\rm div}; \Omega_m)~|~\mathbf{w}\cdot \mathbf{n}_m=0 \ \text{on}\ \Gamma_{m}\}, \non\\
\mathbf{H}_{m,\mathrm{div}}&:=&\{\mathbf{w}\in\mathbf{H}_{m,0}~|~\nabla \cdot\mathbf{w}=0\},\non\\
 X_m&:=& H^1(\Omega_m) \cap L^2_0(\Omega_m).\non
 \eea
 Here $L^2_0(\Omega_m)$ is a subspace of $L^2$ whose elements are of mean zero.
We denote $(\cdot, \cdot)_c$, $(\cdot, \cdot)_m$ the inner products on the spaces $L^2(\Omega_c)$, $L^2(\Omega_m)$, respectively (also for the corresponding vector spaces). The inner product on $L^2(\Omega)$ is simply denoted by $(\cdot, \cdot)$. Then it is clear that
$$
(u,v)=(u_m,v_m)_m+(u_c,v_c)_c, \quad \|u\|_{L^2(\Omega)}^2=\|u_m\|_{L^2(\Omega_m)}^2+\|u_c\|_{L^2(\Omega_c)}^2,
$$
where $u_m:=u|_{\Omega_m}$ and $u_c:=u|_{\Omega_c}$. We will suppress the dependence on the domain in the $L^2$ norm if there is no ambiguity. We also denote $H'$ the dual space of $H$ with the duality induced by the $L^2$ inner product.

Below we give the definition of the weak formulation of the CHSD system in 2D. The weak formulation in 3D can be defined similarly with slight changes in time integrability of the functions.
\begin{definition}\label{defweak} 
Suppose that $d=2$ and $T>0$ is arbitrary. 
We consider the initial data $\vp_{0}\in H^1(\Omega), \ub_c(0) \in \mathbf{H}_{c,div}, \ub_m(0) \in \mathbf{H}_{m,div}$. The functions $(\ub_c, P_c, \mathbf{u}_m, P_m, \vp, \mu)$ with the following properties
\bea
&  \ub_c\in  L^\infty(0, T; \mathbf{L^2}(\Omega_c)) \cap L^2(0, T; \mathbf{H}_{c,0}),  \frac{\partial \ub_c}{\partial t} \in L^{\frac{4}{3}}(0,T; (\mathbf{H}_{c,0})^\prime),  \\
& \ub_m\in L^\infty(0, T; \mathbf{L^2}(\Omega_m)) \cap L^2(0, T; \mathbf{H}_{m,0}),\frac{\partial \ub_m}{\partial t} \in L^{\frac{4}{3}}(0,T; (\mathbf{H}_{m,0})^\prime), \label{regpm}\\
& P_c \in L^\frac43(0,T; L^2(\Omega_c)), \quad P_m \in L^\frac43(0,T; X_m), \\
& \vp\in L^\infty(0,T; H^1(\Omega))\cap L^2(0,T; H^3(\Omega)), \vp_t \in  L^2(0;T; (H^1(\Omega))'),\\
& \mu\in L^2(0, T; H^1(\Omega)),
\eea
is called a finite energy weak solution of the CHSD system \eqref{HSCH-NSCH1}--\eqref{IBCi7}, if the following conditions are satisfied:

(1) For any $\mathbf{v}_c\in  \mathbf{H}_{c,0}$ and $q_c\in  L^2(\Omega_c)$,
 \begin{eqnarray}
 &&\langle \partial_t \ub_c,\mathbf{v}_c \rangle_c + 2(\nu(\vp_c)\mathbb{D}(\ub_c),\mathbb{D}(\mathbf{v}_c))_c - (P_c, \nabla \cdot \mathbf{v}_c)_c
 \non\\
&&+\sum_{i=1}^{d-1}\alpha_{BJSJ}\int_{\Gamma_{cm}}\frac{\nu(\vp_m)}{\sqrt{{\rm trace} (\Pi)}}(\ub_c\cdot\btau_i)(\mathbf{v}_c\cdot\btau_i)dS +  \int_{\Gamma_{cm}} P_m (\mathbf{v}_c\cdot \mathbf{n}_{cm}) dS\non\\
&&  + (\nabla \cdot \ub_c, q_c)_c +(\vp_c \nabla \mu(\varphi_c),  \mathbf{v}_c )_c =0.
 \label{weak1}
\end{eqnarray}

(2) For any $\mathbf{v}_m\in  \mathbf{H}_{m,0}$ and $q_m\in  H^1(\Omega_m)$,
\begin{eqnarray}
&&\frac{\rho_0}{\chi} \langle \partial_t \ub_m,\mathbf{v}_m \rangle_m+(\nu(\varphi_m)\Pi^{-1}\ub_m,  \mathbf{v}_m)_m + (\nabla P_m,  \mathbf{v}_m)_m - ( \ub_m, \nabla q_m)_m\non \\
&& +(\vp_m \nabla \mu(\varphi_m),  \mathbf{v}_m )_m -\int_{\Gamma_{cm}} \ub_c \cdot \mathbf{n}_{cm} q_m\, ds=0.
\label{weak2}
\end{eqnarray}

(3) For any $v, \phi\in  H^1(\Omega)$,
\begin{eqnarray}
&&\langle \partial_t\varphi,v)+({\rm M}(\vp) \nabla \mu(\vp),\nabla v)-(\ub \varphi ,\nabla v)=0, \label{weak3}\\
&&\gamma \left[\frac{1}{\epsilon}(f(\varphi),v)+\epsilon(\nabla\varphi,\nabla v) \right]-(\mu(\varphi),\phi)=0. \label{weak4}
\end{eqnarray}

(3)  $\vp|_{t=0}=\vp_{0}(x), \ub_c|_{t=0}=\ub_c(0), \ub_m|_{t=0}=\ub_m(0).$

(4) The finite energy solution satisfies the energy inequality
\be
\mathcal{E}(t)  +\int_s^t\mathcal{D}(\tau) d\tau \leq \mathcal{E}(s), \label{Energyinq}
\ee
for all $t\in [s,T)$ and almost all $s\in [0,T)$ (including $s=0$), where the total energy $\mathcal{E}$ is given by \eqref{totenergy}.

\end{definition}
We note that the Darcy pressure $P_m$ and the Stokes pressure $P_c$ are uniquely determined only up to a common constant in the CHSD system \eqref{HSCH-NSCH1}--\eqref{IBCi7}. In the Definition \ref{defweak}, we require $P_m \in X_m$ so that it is of mean zero and uniquely determined. Then the Stokes pressure is uniquely determined in view of the interface boundary condition \eqref{IBCi6}. Hence in the weak formulation we only impose $P_c \in L^2(\Omega_c)$.
We refer to \cite{HWW2014} for the study of the  existence  of such a weak solution for a similar problem.

\section{The numerical schemes}

Let $\tau >0$ be a time step size and set $t^k=k\tau$ for $0\leq k \leq K=[T/\tau]$.
Let $\mathcal{T}^h_c$ ($\mathcal{T}^h_m$) be a quasi-uniform triangulation of the domain $\Omega_c$ ($\Omega_m$ resp.) of mesh size $h$. In addition, we assume that the triangulations $\mathcal{T}^h_c$ and $\mathcal{T}^h_m$ coincide on the interface $\Gamma_{cm}$ in the sense that triangles in $\Omega_m$ and $\Omega_c$ share the same edges along $\Gamma_{cm}$. Then $\mathcal{T}^h:= \mathcal{T}^h_c \cup \mathcal{T}^h_m$ forms a triangulation of the domain $\Omega$.   Let  $Y_h$ denote a finite element approximation of  $H^1(\Omega)$  based on the triangulation $\mathcal{T}_h$. Typical examples of $Y_h$ include
\begin{align*}
Y_h=\{v_h \in C(\bar{\Omega})\big| v_h |_K \in P_r(K), \forall K \in \mathcal{T}_h\},
\end{align*}
where $P_r(K)$ is the space of polynomials of degree less than or equal to $r$ on the triangle $K$. 
Denote by $\mathbf{X}_c^h$ the finite element approximation of $\mathbf{H}_{c,0}$, and  by $M_c^h$ the finite element approximation of $L^2(\Omega_c)$. Note that we did not impose the condition of mean zero on the space $M_c^h$. This is consistent with the Definition \ref{defweak}.
We assume that $\mathbf{X}_c^h$ and $M_c^h$ are stable approximation spaces for Stokes velocity and pressure in the sense that
\begin{align}\label{inf-supS}
\sup_{\mathbf{v}_h \in \mathbf{X}_c^h} \frac{(\nabla \cdot \mathbf{v}_h,  q_h)_c}{||\mathbf{v}_h||_{H^1}} \geq c||q_h||_{L^2}, \quad \forall q_h \in M_c^h.
\end{align} 
The validity of such an inf-sup condition for some standard finite element spaces can be found in \cite{LSY2002}. The classical P2-P0,   Taylor-Hood finite element spaces and the Mini finite element spaces    are commonly adopted in practice for $\mathbf{X}_c^h$ and $M_c^h$, cf. \cite{LSY2002, GiRa1986}. Similarly, one can define the finite element spaces $\mathbf{X}_m^h$ (finite element subspace of $\mathbf{H}_{m,0}$) and $M_m^h$ (finite element subsapce of $X_m$) for the Darcy velocity and pressure. We also assume 
$\mathbf{X}_m^h$ and $M_m^h$ are stable satisfying
\begin{align}\label{inf-supD}
\sup_{\mathbf{v}_h \in \mathbf{X}_m^h} \frac{( \mathbf{v}_h, \nabla q_h)_m}{||\mathbf{v}_h||_{L^2}} \geq c||q_h||_{L^2}, \quad \forall q_h \in M_m^h.
\end{align}
We remark that the Taylor-Hood finite element spaces satisfy the above condition.
\begin{remark}
The inf-sup condition \eqref{inf-supD} for Darcy equation is not the standard one. We remark that it holds for some common finite element spaces. For example, if $\nabla(M_m^h)\subset\mathbf{V}_m^h$, then we can can take $\mathbf{v}_h=\nabla q_h$. Therefore, 
$$
   \frac{( \mathbf{v}_h, \nabla q_h)_m}{||\mathbf{v}_h||_{L^2}}=\|\nabla q_h\|_{L^2}\ge c\|q_h\|_{L^2},
$$
here we have used Poincar\'{e} inequality in the last inequality since we require $q_h\in L^2_0(\Omega_m)$.  Another case is to use mixed finite element for the Stokes equation with continuous pressure approximation, such as Taylor-Hood element, and we can choose $ \mathbf{v}_h\cdot \mathbf{n}=0$ with $\mathbf{n}$ the unit outer normal of $\partial \Omega_m$, then $( \mathbf{v}_h, \nabla q_h)_m=-(\nabla \cdot  \mathbf{v}_h, q_h)_m$, and the 
inf-sup condition  \eqref{inf-supD} can be obtained by using the standard inf-sup condition for the Stokes equation.
\end{remark}

Before we describe our unconditionally stable and decoupled numerical schemes, we point out that a fully decoupled numerical scheme for solving the CHSD model can be constructed easily, for instance, by treating  the velocity in the Cahn-Hilliard equation \eqref{weak3} and the interface boundary conditions in Eqs. \eqref{weak1} and \eqref{weak2} explicitly. The explicit treatment of velocity in the phase field fluid models have been analyzed carefully, in \cite{KSW2008} for the case of Cahn-Hilliard-Navier-Stokes equations, and in \cite{HaWa2016}  for the case of Cahn-Hilliard-Darcy model. And decoupled schemes using extrapolation in time for interface boundary conditions have been proposed and analyzed for single phase Stokes-Darcy system, see \cite{CGSW2013, CGSW2015} and references therein. However, in the setting of CHSD model, it seems that this type of decoupling strategy does not lead to unconditional stability. It is our aim here to design unconditionally stable, energy stable in particular,  and decoupled numerical schemes for solving the CHSD model.

\subsection{An energy stable scheme (PD) decoupling the order parameter and the velocity}
Here we present an unconditionally stable numerical scheme that decouples the computation of the Cahn-Hilliard equation from that of the fluid equations (Stokes-Darcy system). We employ a fractional-step method for realizing the decoupling. An intermediate velocity driven solely by the capillary force is used in the time-discretization of the Cahn-Hilliard equation, see Eqs. \eqref{InterM1} and \eqref{InterM2} below. Hence, upon substitution, the velocity equations are completely decoupled from the equations for the order parameter. In the context of phase field models, this idea of fractional step method is first applied in solving the Cahn-Hilliard-Navier-Stokes equations, cf. \cite{Minjeaud2013, ShYa2015}. We point out that the Stokes equations are still coupled with the Darcy equations in the scheme (PD).

We present the scheme (PD) for solving the CHSD model \eqref{HSCH-NSCH1}--\eqref{IBCi7} as follows: 

\noindent Step 1: Cahn-Hilliard equation: find $\varphi_h^{k+1} \in Y_h$ and $\mu_h^{k+1} \in Y_h$ such that for any $v_h, \phi_h \in  Y_h$,
\begin{eqnarray}
&&( \delta_t\varphi^{k+1}_h,v_h)+({\rm M}(\vp^k_h) \nabla \mu^{k+1}_h,\nabla v_h)-(\baru^{k+1}_h \varphi^k_h ,\nabla v_h)=0, \label{3Fweak3}\\
&&\gamma \left[\frac{1}{\epsilon}(f(\varphi^{k+1}_h, \varphi^k_h),\phi_h)+\epsilon(\nabla\varphi^{k+1}_h,\nabla \phi_h) \right]-(\mu^{k+1}_h,\phi_h)=0, \label{3Fweak4}
\end{eqnarray}
with $f(\varphi^{k+1}_h, \varphi^k_h)=(\varphi^{k+1}_h)^3-\varphi^k_h$, and $\delta_t$ denoting the backward difference quotient operator $\delta_t\varphi^{k+1}_h:= \frac{\varphi^{k+1}_h-\varphi^{k}_h}{\tau}$.
Here the intermediate velocity $\baru^{k+1}_h$ in Eq. \eqref{3Fweak3} is defined as
\begin{eqnarray}
\baru ^{k+1}=\left\{
\begin{aligned}
&\baru_{m,h}^{k+1}, \quad x \in \Omega_m, \\
&\baru_{c,h}^{k+1}, \quad x \in \Omega_c,
\end{aligned}
\right.
\end{eqnarray}
where $\baru_{m,h}^{k+1}$ and $\baru_{c,h}^{k+1}$ are defined through the following equations
\begin{align}
& \frac{\rho_0}{\chi}\frac{\baru_{m,h}^{k+1}-\ub_{m,h}^k}{\tau}+\varphi_{m,h}^{k}\nabla \mu_{m,h}^{k+1}=0, \label{InterM1} \\
& \rho_0 \frac{\baru_{c,h}^{k+1}-\ub_{c,h}^k}{\tau}+\varphi_{c,h}^{k}\nabla \mu_{c,h}^{k+1}=0. \label{InterM2}
\end{align}

\noindent Step 2: Stokes equation: find $\ub_{c, h}^{k+1} \in \mathbf{X}_{c}^h$ and $P_{c,h}^{k+1} \in M_c^h$ such that for any $\mathbf{v}_{c,h}\in  \mathbf{X}_{c}^h$ and $q_{c,h}\in  M_c^h$,
 \begin{eqnarray}
 &&\rho_0( \delta_t\ub_{c,h}^{k+1},\mathbf{v}_{c,h} )_c  +a_c(\ub_{c,h}^{k+1},\mathbf{v}_{c,h}\big)+ b_c(\mathbf{v}_{c,h}, P_{c,h}^{k+1})
 +  \int_{\Gamma_{cm}} P_{m,h}^{k+1} (\mathbf{v}_{c,h}\cdot \mathbf{n}_{cm}) dS\non\\
&& - b_c(\mathbf{u}^{k+1}_{c,h}, q_{c,h}) +(\vp_{c,h}^k \nabla \mu_{c,h}^{k+1},  \mathbf{v}_{c,h} )_c =0,
 \label{3Fweak1}
\end{eqnarray}
where
\begin{align}
& a_c(\ub_{c,h}^{k+1},\mathbf{v}_{c,h})=2(\nu(\vp_{c,h}^k)\mathbb{D}(\ub_{c,h}^{k+1}),\mathbb{D}(\mathbf{v}_{c,h}))_c \non \\
& +\sum_{i=1}^{d-1}\alpha_{BJSJ}\int_{\Gamma_{cm}}\frac{\nu(\vp_{c,h}^k)}{\sqrt{{\rm trace} (\Pi)}}(\ub_{c,h}^{k+1}\cdot\btau_i)(\mathbf{v}_{c,h}\cdot\btau_i)dS, \label{AC}\\
& b_c(\mathbf{v}_{c,h}, q_{c,h})=-(\nabla \cdot \mathbf{v}_{c,h}, q_{c,h})_c. \label{BC}
\end{align}

\noindent And Darcy equation: find $\ub_{m, h}^{k+1} \in \mathbf{X}_{m}^h$ and $P_{m,h}^{k+1} \in M_m^h$ such that for any $\mathbf{v}_{m,h}\in  \mathbf{X}_{m}^h$ and $q_{m,h}\in  M_m^h$,
\begin{eqnarray}
&&\frac{\rho_0}{\chi} (\delta_t\ub_{m,h}^{k+1},\mathbf{v}_{m,h} )_m +a_m\big(\mathbf{u}_{m,h}^{k+1},\mathbf{v}_{m,h}\big) 
+b_m(\mathbf{v}_{m,h}, P_{m,h}^{k+1})  +(\vp_{m,h}^k \nabla \mu_{m,h}^{k+1},  \mathbf{v}_{m,h} )_m \non \\
&&-\int_{\Gamma_{cm}} \ub_{c,h}^{k+1} \cdot \mathbf{n}_{cm} q_{m,h}\, ds- b_m(\ub_{m,h}^{k+1},  q_{m,h})=0,
\label{3Fweak2}
\end{eqnarray}
where
\begin{align}
& a_m(\mathbf{u}_{m,h}^{k+1},\mathbf{v}_{m,h})=\big(\nu(\varphi_{m,h}^{k})\Pi^{-1}\mathbf{u}_{m,h}^{k+1}, \mathbf{v}_{m,h}\big)_m, \label{AM}\\
& b_m(\mathbf{v}_{m,h}, q_{m,h})=( \mathbf{v}_{m,h}, \nabla q_{m,h})_m. \label{BM}
\end{align}

The decoupling of the Cahn-Hilliard equation and Stokes-Darcy system is realized through 
a fractional step method. For instance, Eqs. \eqref{InterM1} and \eqref{3Fweak2} amount to solving the Darcy system \eqref{HSCH-NSCH4}-\eqref{HSCH-NSCH5} via the following temporal splitting algorithm, suppressing the spatial discretization,
\begin{align}
&\frac{\rho_0}{\chi}\frac{\baru_{m}^{k+1}-\ub_{m}^k}{\tau}+\varphi_{m}^{k}\nabla \mu_{m}^{k+1}=0, \label{InterS1} \\
&\frac{\rho_0}{\chi}\frac{\ub_{m}^{k+1}-\baru_{m}^{k+1}}{\tau} + \nu(\varphi_{m}^{k}){\Pi}^{-1}\mathbf{u}_{m}^{k+1}
+\nabla P_{m}^{k+1}=0, \label{InterS2} \\
& \nabla \cdot \ub_{m}^{k+1}=0.  \label{InterS3}
\end{align}
 It is clear that the scheme is consistent as $\baru_{m}^{k+1}$ is a first order in-time approximation of $\ub_{m}^k$. Furthermore, the intermediate velocity never appears in the practical computation as one can substitute the definition of $\baru_{m}^{k+1}$ back into the Cahn-Hilliard equation and solve for $\mathbf{u}_{m}^{k+1}$ via Eq. \eqref{3Fweak2}.

To state the energy stability of the scheme (PD), we define a discrete free energy functional
\begin{align*}
E(\varphi_h^k)=\gamma \int_\Omega \left(\frac{1}{\epsilon} F(\varphi_h^k)+ \frac{\epsilon}{2}|\nabla \varphi_h^k|^2 \right)dx,
\end{align*}
 and also a discrete total energy functional
 \begin{align}\label{DisEner}
 \mathcal{E}^{k}= \int_{\Omega_c}\frac{\rho_0}{2}|\ub_{c,h}^{k}|^2 dx+ \int_{\Omega_m}\frac{\rho_0}{2\chi}|\ub_{m,h}^k|^2 dx+ \gamma \int_\Omega \left(\frac{1}{\epsilon} F(\varphi_h^k)+ \frac{\epsilon}{2}|\nabla \varphi_h^k|^2 \right)dx.
 \end{align}
One can show that the scheme \eqref{3Fweak3}--\eqref{3Fweak2} is unconditionally uniquely solvable and energy-stable, in the sense of the following theorem.
\begin{theorem}\label{PDsta}
The scheme (PD) (Eqs. \eqref{3Fweak3}--\eqref{3Fweak2}) is unconditionally uniquely solvable and mass conservative at each time step. Moreover, 
the scheme (PD)  satisfies a modified energy law
\begin{align}
&\mathcal{E}^{k+1}-\mathcal{E}^k + \tau  a_c(\ub_{c,h}^{k+1},\ub_{c,h}^{k+1})+ \tau ||\sqrt{\nu/\Pi}\ub_{m,h}^{k+1}||_{L^2}^2  \non \\
&\leq -\frac{\gamma \epsilon}{2}||\nabla(\ar1-\var0)||_{L^2}^2 -\frac{\rho_0}{4}||\ub_{c,h}^{k+1}-\ub_{c,h}^k||_{L^2}^2-\frac{\rho_0}{ 4\chi}||\ub_{m,h}^{k+1}-\ub_{m,h}^k||_{L^2}^2, \label{PDEner}
\end{align}
 Thus it is unconditionally energy-stable.
\end{theorem}
\begin{proof}
Note that upon substitution of the intermediate velocities \eqref{InterM1} and \eqref{InterM2} into the Cahn-Hilliard equation \eqref{3Fweak3}, the nonlinear Cahn-Hilliard equations \eqref{3Fweak3}-\eqref{3Fweak4} are completely decoupled from the linear Stokes-Darcy equations   
\eqref{3Fweak1}-\eqref{3Fweak2}. Given $\varphi^k_h, \mathbf{u}^k_h$, Eqs. \eqref{3Fweak3}-\eqref{3Fweak4} can be viewed as a first-order convex-splitting discretization of the Cahn-Hilliard equation with known source terms. Thus the unique solvability of the Cahn-Hilliard part can be established by following a gradient flow  argument, cf. \cite{KSW2008, Wise2010, Shen2012}.  See also \cite{HaWa2015, HWW2014} for an alternative proof exploiting the property of  monotonicity in convex-splitting schemes. Once $\mu_h^{k+1}$ is known, Eqs. \eqref{3Fweak1}-\eqref{3Fweak2} become linear equations for velocity and pressure. 
Its unique solvability can be established the same way as in the classical mixed formulation for Stokes equation, cf. \cite{GiRa1986} using the inf-sup conditions. For completeness, we give the details of an alternative argument that does not rely on the inf-sup condition for the coupled Stokes-Darcy system explicitly here. As Eqs. \eqref{3Fweak1}-\eqref{3Fweak2} define a finite linear system for $\mathbf{u}^{k+1}_{c,h}, P_{c,h}^{k+1}, \mathbf{u}^{k+1}_{m,h}, P_{m,h}^{k+1}$, we only need to show that solutions are unique. \footnote{Invertibility is equivalent to a trivial kernel for a square matrix.} Suppose there are two solutions, and define their differences by $\E_c^{u}, E_c^p, \E_m^u, E_m^p$ respectively. Then the differences satsify $\forall \mathbf{v}_{c,h}\in  \mathbf{X}_{c}^h, q_{c,h}\in  M_c^h$ and $\mathbf{v}_{m,h}\in  \mathbf{X}_{m}^h, q_{m,h}\in  M_m^h$,
\begin{align}
 &\frac{\rho_0}{\tau}( \E_c^u,\mathbf{v}_{c,h} )_c  +a_c(\E_c^u,\mathbf{v}_{c,h}) +b_c(\mathbf{v}_{c,h}, E_c^p)  -b_c(  \E_c^u, q_{c,h})
 \non\\
& +  \int_{\Gamma_{cm}} E_m^p (\mathbf{v}_{c,h}\cdot \mathbf{n}_{cm}) dS =0,  
 \label{PFweak1} \\
 &\frac{\rho_0}{\chi \tau}(\E_m^u,\mathbf{v}_{m,h} )_m +a_m (\E_m^u, \mathbf{v}_{m,h})
+b_m(\mathbf{v}_{m,h}, E_m^p ) - b_m(\E_m^u, q_{m,h}) \non \\
&-\int_{\Gamma_{cm}} \E_c^u \cdot \mathbf{n}_{cm} q_{m,h}\, ds=0.
\label{PFweak2}
\end{align}
Taking $\mathbf{v}_{c,h}=\E_c^u, q_{c,h}=E_c^p$ in Eq. \eqref{PFweak1} and $\mathbf{v}_{m,h}=\E_m^u, q_{m,h}=E_m^p$ in Eq. \eqref{PFweak2}, and adding the equations together, one obtains
\begin{align*}
&\frac{\rho_0}{\tau}( \E_c^u,\E_c^u )_c  +a_c(\E_c^u,\E_c^u)+\frac{\rho_0}{\chi \tau}(\E_m^u,\E_m^u )_m +a_m (\E_m^u, \E_m^u)=0.
\end{align*}
Hence $\E_c^u=\E_m^u=0$ and Eqs. \eqref{PFweak1}, \eqref{PFweak2} reduce to 
\begin{align}
& b_c(\mathbf{v}_{c,h}, E_c^p)+ \int_{\Gamma_{cm}} E_m^p (\mathbf{v}_{c,h}\cdot \mathbf{n}_{cm}) dS =0, \label{Rbc} \\
&b_m(\mathbf{v}_{m,h}, E_m^p )=0.
\end{align}
It then follows from the inf-sup conditions \eqref{inf-supD} that $E_m^p=0$. Eq. \eqref{Rbc}  can be written as 
\begin{align}\label{RRbc}
b_c(\mathbf{v}_{c,h}, E_c^p-C)+ C\int_{\Gamma_{cm}} \mathbf{v}_{c,h} \cdot \mathbf{n}_{cm} dS=0,
\end{align}
with the constant $C=\frac{1}{|\Omega_c|} \int_{\Omega_c} E_c^p dx.$ Now if $\mathbf{v}_{c,h}=0$ on $\Gamma_{cm}$, then the classical inf-sup condition for Stokes equations (satisfied, for instance, by Taylor-Hood finite element spaces) implies that $E_c^p=C$. Combining this with Eq. \eqref{RRbc} further yields $E_c^p=C=0$.
Therefore the two solutions must be the same. Thus the scheme (PD) is unconditionally uniquely solvable at each time step.

Now we show that the modified energy law \eqref{PDEner} holds.
Owing to the convexity, one can establish the elementary inequality
\begin{align}\label{NonIne}
F(\varphi^{k+1}_h)- F(\varphi^{k}_h)\leq  f(\varphi^{k+1}_h, \varphi^k_h)(\varphi^{k+1}_h-\varphi^k_h),
\end{align}
where one may recall $F(\varphi)=\frac{1}{4}(\varphi^2-1)^2$ and $f(\phi, \varphi)=\phi^3-\varphi$.
Taking the test function $v_h=\tau \mu_h^{k+1}$ in Eq. \eqref{3Fweak3} and $\phi_h=\varphi^{k+1}_h-\varphi^{k}_h$ in Eq. \eqref{3Fweak4}, and adding the results together, one obtains by virtue of the inequality \eqref{NonIne}
\begin{align}\label{CH-E1}
E(\varphi_h^{k+1})-E(\var0)+\tau ||\sqrt{M}\nabla \mu_h^{k+1}||_{L^2}^2 - \tau(\baru^{k+1}_h \varphi^k_h ,\nabla \mu_h^{k+1}) \leq -\frac{\gamma \epsilon}{2}||\nabla(\ar1-\var0)||_{L^2}^2.
\end{align}

Next, it follows from Eq. \eqref{InterM1} and Eq. \eqref{InterM2} that
\begin{align}\label{InterMIne}
&\frac{\rho_0}{2\chi}\large\{||\baru_{m,h}^{k+1}||^2_{L^2}-||\ub_{m,h}^k||^2_{L^2}+||\baru_{m,h}^{k+1}-\ub_{m,h}^k||^2_{L^2} \large\}+\frac{\rho_0}{2}\large\{||\baru_{c,h}^{k+1}||^2_{L^2}-||\ub_{c,h}^k||^2_{L^2}+||\baru_{c,h}^{k+1}-\ub_{c,h}^k||^2_{L^2} \large\} \nonumber \\
&+ \tau \big(\baru^{k+1}_h \varphi^k_h ,\nabla \mu_h^{k+1}\big)=0.
\end{align}

Take the test function $\mathbf{v}_{c,h}= \tau \mathbf{u}_{c,h}^{k+1}$ and $q_{c,h}= P_{c, h}^{k+1}$ in Eq. \eqref{3Fweak1}, and use Eq. \eqref{InterM2}
\begin{align}\label{S-Ine}
\frac{\rho_0}{2}\large\{||\ub_{c,h}^{k+1}||^2_{L^2}&-||\baru_{c,h}^{k+1}||^2_{L^2}+||\ub_{c,h}^{k+1}-\baru_{c,h}^{k+1}||^2_{L^2} \large\} +  \tau   a_c(\ub_{c,h}^{k+1},\ub_{c,h}^{k+1}) \non\\
&+  \tau \int_{\Gamma_{cm}} P_{m,h}^{k+1} (\mathbf{u}_{c,h}^{k+1}\cdot \mathbf{n}_{cm}) dS =0.
\end{align}

Similarly, by using the intermediate velocity in Eq. \eqref{InterM1}, taking the test functions
$\mathbf{v}_{m,h}=\tau \ub_{m,h}^{k+1}$ and $q_{m,h}= \tau P_{m,h}^{k+1}$ in Eq. \eqref{3Fweak2}, one obtains that
\begin{align}\label{Darcy-E1}
&\frac{\rho_0}{2\chi}\large\{||\ub_{m,h}^{k+1}||^2_{L^2}-||\baru_{m,h}^{k+1}||^2_{L^2}+||\ub_{m,h}^{k+1}-\baru_{m,h}^{k+1}||^2_{L^2} \large\} +\tau ||\sqrt{\nu/\Pi}\ub_{m,h}^{k+1}||_{L^2}^2 \non \\
&-\tau \int_{\Gamma_{cm}} \ub_{c,h}^{k+1} \cdot \mathbf{n}_{cm} P_{m,h}^{k+1}\, dS =0.
\end{align}

Finally summing up Eqs. \eqref{CH-E1}, \eqref{InterMIne}, \eqref{S-Ine} and \eqref{Darcy-E1}, we obtain the modified energy law
\begin{align*}
&\mathcal{E}^{k+1}-\mathcal{E}^k + \tau ||\sqrt{M}\nabla \mu_h^{k+1}||_{L^2}^2+ \tau  a_c(\ub_{c,h}^{k+1},\ub_{c,h}^{k+1})+ \tau ||\sqrt{\nu/\Pi}\ub_{m,h}^{k+1}||_{L^2}^2 \non \\
& \leq -\frac{\gamma \epsilon}{2}||\nabla(\ar1-\var0)||_{L^2}^2 -\frac{\rho_0}{4}||\ub_{c,h}^{k+1}-\ub_{c,h}^k||_{L^2}^2-\frac{\rho_0}{4\chi}||\ub_{m,h}^{k+1}-\ub_{m,h}^k||_{L^2}^2,
\end{align*}
where the elementary inequality $2(a^2+b^2)\ge (a+b)^2$ has been used. 

We also observe that our scheme is mass conservative by simply taking the test function in the phase-field equation to be 1.

This completes the proof. 
\qed
\end{proof}

In the scheme (PD), the Darcy equations are solved in the primitive velocity-pressure formalism Eq. \eqref{3Fweak2}. It is natural to solve the Darcy equations using the pressure alone as the primary variable, since the Darcy pressure (more precisely, the hydraulic head) is of practical importance in applications of flow in porous media. Moreover, there are efficient fast solvers for Poisson equation. By solving for $\ub_{m}^{k+1}$ from Eq. \eqref{InterS2} and substituting the resulting expression into Eq. \eqref{InterS3}, one can solve the Darcy equation \eqref{3Fweak2} (coupled with the Stokes equation \eqref{3Fweak1}) via

\noindent find  $P_{m,h}^{k+1} \in M_m^h$ such that for any $q_{m,h}\in  M_m^h$,
\begin{align}
&\Big(\frac{\tau \Pi \chi}{\rho_0 \Pi+\tau \nu \chi} \nabla P^{k+1}_{m,h}-\frac{\rho_0 \Pi}{\rho_0 \Pi+\tau \nu \chi} \baru^{k+1}_{m,h}, \nabla q_{m,h}\Big)_m- \int_{\Gamma_{cm}} \ub_{c,h}^{k+1} \cdot \mathbf{n}_{cm} q_{m,h}\, dS=0 \label{4Fweak2}.
\end{align}
Then the Darcy velocity at time level $k+1$ is recovered via the projection of the algebraic equation 
\begin{eqnarray}
&&\frac{\rho_0}{\chi} \frac{\ub_{m,h}^{k+1}-\ub_{m,h}^k}{\tau}+ \frac{\nu(\varphi_{m,h}^{k})}{\Pi}\mathbf{u}_{m,h}^{k+1}
+\nabla P_{m,h}^{k+1}  +\vp_{m,h}^k \nabla \mu_{m,h}^{k+1}
=0.
\label{4Fweak5}
\end{eqnarray}

\subsection{A fully decoupled numerical scheme (FD)}
In the scheme \eqref{3Fweak3}--\eqref{3Fweak2}, the Darcy equation is still coupled with the Stokes equation.  We present here a fully decoupled scheme such that the order parameter, the Darcy pressure and the Stokes velocity can be calculated independently while maintaining the desired energy stability.  In the scheme (FD) below,  we consider using the domain decomposition method to decouple the Darcy-Stokes system, which has been  studied intensively for single phase flow, for instance,  in \cite{CR08,DMQ2002, DiQu2003, DiQu2009, LSY2002, CGHW2010, CGW2010, CGHW2011, ChGHW2011, CGSW2013,  CR09, CR12, CGR13}.
In the following scheme (FD), the solution to Darcy system is firstly computed,
and then the solution of Stokes system is computed after $P_{m,h}^{k+1}$ is obtained.

The fully decoupled scheme (FD) reads as follows: \\
\noindent Step 1: Cahn-Hilliard equation: find $\varphi_h^{k+1} \in Y_h$ and $\mu_h^{k+1} \in Y_h$ such that for any $v_h, \phi_h \in  Y_h$,
\begin{eqnarray}
&&( \delta_t\varphi^{k+1}_h,v_h)+({\rm M}(\vp^k_h) \nabla \mu^{k+1}_h,\nabla v_h)-(\baru^{k+1}_h \varphi^k_h ,\nabla v_h)=0, \label{5Fweak3}\\
&&\gamma \left[\frac{1}{\epsilon}(f(\varphi^{k+1}_h, \varphi^k_h),\phi_h)+\epsilon(\nabla\varphi^{k+1}_h,\nabla \phi_h) \right]-(\mu^{k+1}_h,\phi_h)=0, \label{5Fweak4}
\end{eqnarray}
where $f(\varphi^{k+1}_h, \varphi^k_h)=(\varphi^{k+1}_h)^3-\varphi^k_h$, and the intermediate velocity $\baru^{k+1}_h$ in Eq. \eqref{5Fweak3} is defined as
\begin{eqnarray}\label{DefVel}
\baru^{k+1}_h=\left\{
\begin{aligned}
&\baru_{m,h}^{k+1}, \quad x \in \Omega_m, \\
&\baru_{c,h}^{k+1}, \quad x \in \Omega_c.
\end{aligned}
\right.
\end{eqnarray}
Here $\baru_{m,h}^{k+1}$ and $\baru_{c,h}^{k+1}$ are defined through the following equations
\begin{align}
& \frac{\rho_0}{\chi}\frac{\baru_{m,h}^{k+1}-\ub_{m,h}^k}{\tau}+\varphi_{m,h}^{k}\nabla \mu_{m,h}^{k+1}=0, \label{InterMV11} \\
& \rho_0 \frac{\baru_{c,h}^{k+1}-\ub_{c,h}^k}{\tau}+\varphi_{c,h}^{k}\nabla \mu_{c,h}^{k+1}=0. \label{InterMV22}
\end{align}

\noindent Step 2: Darcy equation:  find $\ub_{m, h}^{k+1} \in \mathbf{X}_{m}^h$ and $P_{m,h}^{k+1} \in M_m^h$ such that for any $\mathbf{v}_{m,h}\in  \mathbf{X}_{m}^h$ and $q_{m,h}\in  M_m^h$,
\begin{eqnarray}
&&\left( \frac{\rho_0}{\chi}\frac{\ub_{m,h}^{k+1}-\ub_{m,h}^k}{\tau}+\frac{\nu(\varphi_{m,h}^{k})}{\Pi}\mathbf{u}_{m,h}^{k+1}
+ \nabla P_{m,h}^{k+1}  +\vp_{m,h}^k \nabla \mu_{m,h}^{k+1},\mathbf{v}_{m,h}\right)_m
=0.
\label{5Fweak2} \\
&& \beta\tau\left(\nabla  P_{m,h}^{k+1},\nabla q_{m,h}\right)_m-\Big(\ub_{m,h}^{k+1}, \nabla q_{m,h}\Big)_m- \int_{\Gamma_{cm}} \ub_{c,h}^{k} \cdot \mathbf{n}_{cm} q_{m,h}\, dS=0. \label{5Fweak5}
\end{eqnarray}

\noindent Step 3: Stokes equation: find $\ub_{c, h}^{k+1} \in \mathbf{X}_{c}^h$ and $P_{c,h}^{k+1} \in M_c^h$ such that for any $\mathbf{v}_{c,h}\in  \mathbf{X}_{c}^h$ and $q_{c,h}\in  M_c^h$,
\begin{eqnarray}
 &&\rho_0( \delta_t\ub_{c,h}^{k+1},\mathbf{v}_{c,h} )_c  + a_c(\ub_{c,h}^{k+1},\mathbf{v}_{c,h}) + b_c(\mathbf{v}_{c,h}, P_{c,h}^{k+1})
+  \int_{\Gamma_{cm}} P_{m,h}^{k+1} (\mathbf{v}_{c,h}\cdot \mathbf{n}_{cm}) dS\non\\
&& - b_c( \ub_{c,h}^{k+1}, q_{c,h})_c +(\vp_{c,h}^k \nabla \mu_{c,h}^{k+1},  \mathbf{v}_{c,h} )_c =0,  \label{5Fweak1}
\end{eqnarray}
where we one may recall the definition of $a_c$ and $b_c$ from \eqref{AC} and \eqref{BC}.

Note that a first order stabilization term has been added to the  equation \eqref{5Fweak5}.
 The parameter $\beta>0$ will be a suitable constant that
only depends on the geometry of $\Omega_m$ and $\Omega_c$.
By using the domain decomposition,  the Darcy-Stokes system can be solved in a decoupled manner and legacy codes can be used in each of those steps. The scheme can also be regarded
as one of implicit-explicit(IMEX) schemes. Let us define the interface term:
\begin{equation}
{\mathcal{E}_{\Gamma}} = - \int_{\Gamma_{cm}}(\ub_{c,h}^{k+1}-\ub_{c,h}^{k}) \cdot \mathbf{n}_{cm} P_{m,h}^{k+1}\, dS,
\end{equation}
and we need the following lemma to bound this term.

\begin{lemma}\label{Lemint1}
Suppose $\mathbf{v}_{c,h}\in \mathbf{X}_c^h$ satisfies
\be
  (\nabla\cdot \mathbf{v}_{c,h},q_{c,h})_c=0, \quad \forall q_{c,h}\in M_c^h, \label{vdiv}
\ee
then
\begin{equation}
\left| \int_{\Gamma_{cm}}\mathbf{v}_{c,h} \cdot \mathbf{n}_{cm} w \, dS\right|\le C\|\nabla w\|_{L^2(\Omega_m)}\|\mathbf{v}_{c,h}\|_{L^2(\Omega_c)}, \forall w \in  X_m.
 \end{equation}
\end{lemma}
\begin{proof} Since $w \in  X_m:=H^1(\Omega_m)\cap L_0^2(\Omega_m)$, there exists  an extension $W\in {H}^1(\Omega)$ such that: $W|_{\Omega_m}=w$,  $ W|_{\Gamma_{cm}}=w|_{\Gamma_{cm}}$ and
\begin{equation}
  \| W\|_{H^1(\Omega)}\le C\|\nabla w\|_{L^2(\Omega_m)}
\end{equation}
where $C$ is a constant independent of 
   Then by Green's formula,
\begin{equation}
\int_{\Gamma_{cm}}\mathbf{v}_{c,h} \cdot \mathbf{n}_{cm} w \, dS=(\nabla \cdot \mathbf{v}_{c,h}, W)_c+(\nabla W, \mathbf{v}_{c,h})_c.
\end{equation}
Let $W_h$ be the $L^2$ projection of $W$ on the space $M_c^h\cap L_0^2(\Omega_c)$. We have
\begin{equation}
   \|W-W_h\|_{L^2}\le Ch\| W\|_{H^1}.
\end{equation}
Consequently, thanks to the inverse estimate,
\begin{equation}
|(\nabla \cdot \mathbf{v}_{c,h}, W)_c|=|(\nabla \cdot \mathbf{v}_{c,h}, W-W_h)_c|\le Ch\|\nabla \cdot \mathbf{v}_{c,h}\|_{L^2}\|W\|_{H^1}\le C \| \mathbf{v}_{c,h}\|_{L^2}\|W\|_{H^1}.
\end{equation}
Now by using the extension theorem,
\begin{equation}
\int_{\Gamma_{cm}}\mathbf{v}_{c,h} \cdot \mathbf{n}_{cm} w \, dS\le  C \| \mathbf{v}_{c,h}\|_{L^2}\| W\|_{H^1}\le  C \| \mathbf{v}_{c,h}\|_{L^2(\Omega_c)} \|\nabla w\|_{L^2(\Omega_m)}.
\end{equation}
This proves the lemma.
\qed
\end{proof}

 Following the same argument as in the proof of Theorem \ref{PDsta}, we have the following solvability and stability result:
\begin{theorem}
\label{S2stab-2lem}
The scheme (FD) \eqref{5Fweak3}-\eqref{5Fweak1} is unconditionally uniquely solvable and mass conservative.
There exists a constant $\beta$ depending only on the geometry and $\rho_0$ such that  the following modified energy law holds
\begin{align} 
&\mathcal{E}^{k+1} + \tau ||\sqrt{M}\nabla \mu_h^{k+1}||_{L^2}^2+ \tau  a_c(\ub_{c,h}^{k+1},\ub_{c,h}^{k+1})+ \tau ||\sqrt{\nu/\Pi}\ub_{m,h}^{k+1}||_{L^2}^2 + \frac{\beta\tau^2}{2}||\nabla P_{m,h}^{k+1}||^2_{L^2}\non\\
&
+ \frac{\tau^2}{4\rho_0}\left(\chi\|\varphi^k_{m,h}\nabla \mu_{m,h}^{k+1}\|^2_{L^2}+\|\varphi^k_{c,h}\nabla \mu_{c,h}^{k+1}\|^2_{L^2}\right)\non\\
& \leq  \mathcal{E}^k-\frac{\rho_0}{6\chi} ||\ub_{m,h}^{k+1}-\ub_{m,h}^{k}||^2_{L^2}-\frac{\rho_0}{12}||\ub_{c,h}^{k+1}-\ub_{c,h}^{k}||^2_{L^2}-\frac{\gamma \epsilon}{2}||\nabla(\ar1-\var0)||_{L^2}^2 . \label{EDDM2}
\end{align}
\end{theorem}
\begin{proof}
Note that the discretization of the Cahn-Hilliard equations is the same in the fully decoupled scheme (FD) and in the scheme (PD). Hence the inequality \eqref{CH-E1} holds for the Eqs. \eqref{5Fweak3} and \eqref{5Fweak4}, which we copy here for completeness
\begin{align}\label{CH-E11}
E(\varphi_h^{k+1})-E(\var0)+\tau ||\sqrt{M}\nabla \mu_h^{k+1}||_{L^2}^2 - \tau(\baru^{k+1}_h \varphi^k_h ,\nabla \mu_h^{k+1}) \leq -\frac{\gamma \epsilon}{2}||\nabla(\ar1-\var0)||_{L^2}^2.
\end{align}
By using the definition of $\baru^{k+1}_h$ in Eqs. \eqref{DefVel}, \eqref{InterMV11} and \eqref{InterMV22},  one can rewrite Eq. \eqref{CH-E11} as follows
\begin{align}\label{CH-E1-2}
E(\varphi_h^{k+1})-E(\var0)+\tau ||\sqrt{M}\nabla \mu_h^{k+1}||_{L^2}^2 - \tau(\ub^{k}_h \varphi^k_h ,\nabla \mu_h^{k+1})+ W_1 \leq -\frac{\gamma \epsilon}{2}||\nabla(\ar1-\var0)||_{L^2}^2,
\end{align}
where
\begin{equation}\label{Wl}
W_1=\frac{\tau^2}{\rho_0}\left(\chi\|\varphi^k_{m,h}\nabla \mu_{m,h}^{k+1}\|^2+\|\varphi^k_{c,h}\nabla \mu_{c,h}^{k+1}\|^2\right).
\end{equation}
Take the test function $\mathbf{v}_{c,h}= \tau \mathbf{u}_{c,h}^{k+1}$ and $q_{c,h}= P_{c, h}^{k+1}$ in Eq. \eqref{5Fweak1}, 
\begin{align}\label{S-Ine-2}
&\frac{\rho_0}{2}\large\{||\ub_{c,h}^{k+1}||^2_{L^2}-||\ub_{c,h}^{k}||^2_{L^2}+||\ub_{c,h}^{k+1}-\ub_{c,h}^{k}||^2_{L^2} \large\} +  \tau   a_c(\ub_{c,h}^{k+1},\ub_{c,h}^{k+1}) \non\\
&+  \tau \int_{\Gamma_{cm}} P_{m,h}^{k+1} (\mathbf{u}_{c,h}^{k+1}\cdot \mathbf{n}_{cm}) dS +\tau(\ub^{k+1}_{c,h} \varphi^k_{c,h} ,\nabla \mu_{c,h}^{k+1})_c =0.
\end{align}

Testing Eq. \eqref{5Fweak2} with $\tau \ub_{m,h}^{k+1}$, taking $q_{m,h}= \tau P_{m,h}^{k+1}$ in Eq. \eqref{5Fweak5}, and summing up the results gives us
\begin{align}\label{Darcy-E2}
&\frac{\rho_0}{2\chi}\large\{||\ub_{m,h}^{k+1}||^2_{L^2}-||\ub_{m,h}^{k}||^2_{L^2}+||\ub_{m,h}^{k+1}-\ub_{m,h}^{k}||^2_{L^2} \large\} +\tau ||\sqrt{\nu/\Pi}\ub_{m,h}^{k+1}||_{L^2}^2 + \beta\tau^2||\nabla P_{m,h}^{k+1}||^2_{L^2}\non \\
&+\tau(\ub^{k+1}_{m,h} \varphi^k_{m,h} ,\nabla \mu_{m,h}^{k+1})_m-\tau \int_{\Gamma_{cm}} \ub_{c,h}^{k} \cdot \mathbf{n}_{cm} P_{m,h}^{k+1}\, dS =0.
\end{align}
Now summing up the  three estimates \eqref{CH-E1-2}, \eqref{S-Ine-2} and \eqref{Darcy-E2}, we have
\begin{align}
&\mathcal{E}^{k+1}-\mathcal{E}^k + \tau ||\sqrt{M}\nabla \mu_h^{k+1}||_{L^2}^2+ \tau  a_c(\ub_{c,h}^{k+1},\ub_{c,h}^{k+1})+ \tau ||\sqrt{\nu/\Pi}\ub_{m,h}^{k+1}||_{L^2}^2 +W_2  \non \\
&\
+ \beta\tau^2||\nabla P_{m,h}^{k+1}||^2_{L^2}+\frac{\gamma \epsilon}{2}||\nabla(\ar1-\var0)||_{L^2}^2   \leq  - \tau \int_{\Gamma_{cm}}(\ub_{c,h}^{k+1}-\ub_{c,h}^{k}) \cdot \mathbf{n}_{cm} P_{m,h}^{k+1}\, dS, \label{EDDM1}
\end{align}
where
\begin{align*}
W_2=\frac{\rho_0}{2}||\ub_{c,h}^{k+1}-\ub_{c,h}^{k}||^2_{L^2} +\frac{\rho_0}{2\chi}||\ub_{m,h}^{k+1}-\ub_{m,h}^{k}||^2_{L^2}+ W_1
+\tau(\ub^{k+1}_{ h}-\ub^{k}_{ h}, \varphi^k_{h} \nabla \mu_{h}^{k+1}),
\end{align*}
with $W_1$ defined in Eq. \eqref{Wl}.

Applying Young's inequality $(a,b)\le \frac{a^2}{3}+\frac{3b^2}{4}$ to the term $\tau(\ub^{k+1}_{h}-\ub^{k}_{h}, \varphi^k_{h} \nabla \mu_{h}^{k+1})$, one obtains
\begin{eqnarray}
W_2 \geq  \frac{\rho_0}{6}||\ub_{c,h}^{k+1}-\ub_{c,h}^{k}||^2_{L^2} +\frac{\rho_0}{6\chi}||\ub_{m,h}^{k+1}-\ub_{m,h}^{k}||^2_{L^2}+ \frac{1}{4}W_1.
\end{eqnarray}  
By Lemma \ref{Lemint1}, the right-hand side of inequality \eqref{EDDM1} can be bounded as follows
\begin{equation}
\tau|\mathcal{E}_{\Gamma}|\le C\tau\|\nabla P_{m,h}^{k+1}\|_{L^2}\|\ub_{c,h}^{k+1}-\ub_{c,h}^{k}\|_{L^2}
\le\frac{\rho_0}{12}\|\ub_{c,h}^{k+1}-\ub_{c,h}^{k}\|_{L^2}^2+C_1\tau^2\|\nabla P_{m,h}^{k+1}\|_{L^2}.
\end{equation}
If we impose $\beta \ge2  C_1$ which only depends on the geometry of $\Omega_m$, $\Omega_c$ and $\rho_0$, then one has,
\begin{align}
&\mathcal{E}^{k+1}-\mathcal{E}^k + \tau ||\sqrt{M}\nabla \mu_h^{k+1}||_{L^2}^2+ \tau  a_c(\ub_{c,h}^{k+1},\ub_{c,h}^{k+1})+ \tau ||\sqrt{\nu/\Pi}\ub_{m,h}^{k+1}||_{L^2}^2 + \frac{\beta \tau^2}{2}||\nabla P_{m,h}^{k+1}||^2_{L^2}\non\\
&+\frac{\gamma \epsilon}{2}||\nabla(\ar1-\var0)||_{L^2}^2
+\frac{\rho_0}{12}||\ub_{c,h}^{k+1}-\ub_{c,h}^{k}||^2_{L^2} +\frac{\rho_0}{6\chi}||\ub_{m,h}^{k+1}-\ub_{m,h}^{k}||^2_{L^2}+ \frac{1}{4}W_1\non\\
& \leq  0. \non
\end{align} 
Hence we have established the energy inequality \eqref{EDDM2}.

Finally we comment on the unique solvability of the fully decoupled scheme (FD). The unique solvability of the Cahn-Hilliard equation is the same as in the scheme (PD). The Darcy equations \eqref{5Fweak2}--\eqref{5Fweak5} are unconditionally uniquely solvable by the standard energy estimate,  which does not rely on the inf-sup condition. Then the solvability of the Stokes equation \eqref{5Fweak1} is the same as in the proof of Theorem \ref{PDsta}.
The conservation of mass follows from setting the test function in the phase-filed equation to be 1.
This concludes the proof of Theorem \ref{S2stab-2lem}.
\qed
\end{proof}

It is also possible to formulate this fully decoupled scheme utilizing the Darcy pressure as the primary variable in the porous media instead of the velocity and the pressure following the same argument as the one used to derive the Darcy pressure formulation for the partially decoupled scheme presented at the end of the previous subsection.

\section{Numerical experiments}
In this section, we present some numerical examples to show that our numerical schemes can accurately capture the dynamics of two-phase flow in a karst geometry. In the first numerical example, we demonstrate numerically that our schemes are of first order accuracy in time and are long-time stable. The second example illustrates that a droplet passes through the karst system driven by boundary-injection. In the last numerical example, we show that  a lighter bubble rises and penetrates the domain interface due to buoyancy. All the numerical tests are performed using the free software FreeFem++ \cite{Hecht2012}.

\subsection{Convergence and stability}
In the first numerical test, we verify that our schemes are first-order accurate in time.  The computational domain is $[0, 1] \times [-1, 1]$ with the lower half being the conduit and the upper half being the matrix. The approach that we take for the accuracy test is as follows. We calculate a solution using our numerical schemes with sufficiently small $h=0.01$ and $\tau=0.0001$, and view this solution as an accurate one. We then compare the numerical solutions with larger time step-size to this accurate solution and calculate the error measured in $L^2$ norm. Throughout, the celebrated  Taylor-Hood P2--P1 finite elements are employed for the approximation of velocity and pressure, and the P1-P1 pair is used for the discretization of order parameter and chemical potential.  Hence the temporal error is the dominating factor in the overall error.  As an example, we show the results for the fully decoupled scheme (FD), i.e., Eqs. \eqref{5Fweak3}--\eqref{5Fweak1}. The error behavior for the other scheme is similar,  as far as the accuracy is concerned. 

For simplicity, all the parameters appearing in the system \eqref{HSCH-NSCH1}--\eqref{HSCH-NSCH6} are set to be unity. The initial conditions are $\varphi_0= 0.24 \cos(2\pi x) \cos(2\pi y)+ 0.4 \cos(\pi x) \cos(3 \pi y) +1.0$, $\mathbf{u}_0=(-2 \sin^2(\pi x) \sin(2\pi y), 2 \sin(2 \pi x) \sin^2(\pi y))$. The convergence result is shown in Fig. \ref{ConvTest}. The first order convergence rate in time is observed for the variables $\mathbf{u}_c$, $\mathbf{u}_m$, $p_m$, and $\phi$.
\begin{figure}[h!]
\centering
 \includegraphics[width=0.9\textwidth, natwidth=610,natheight=642]{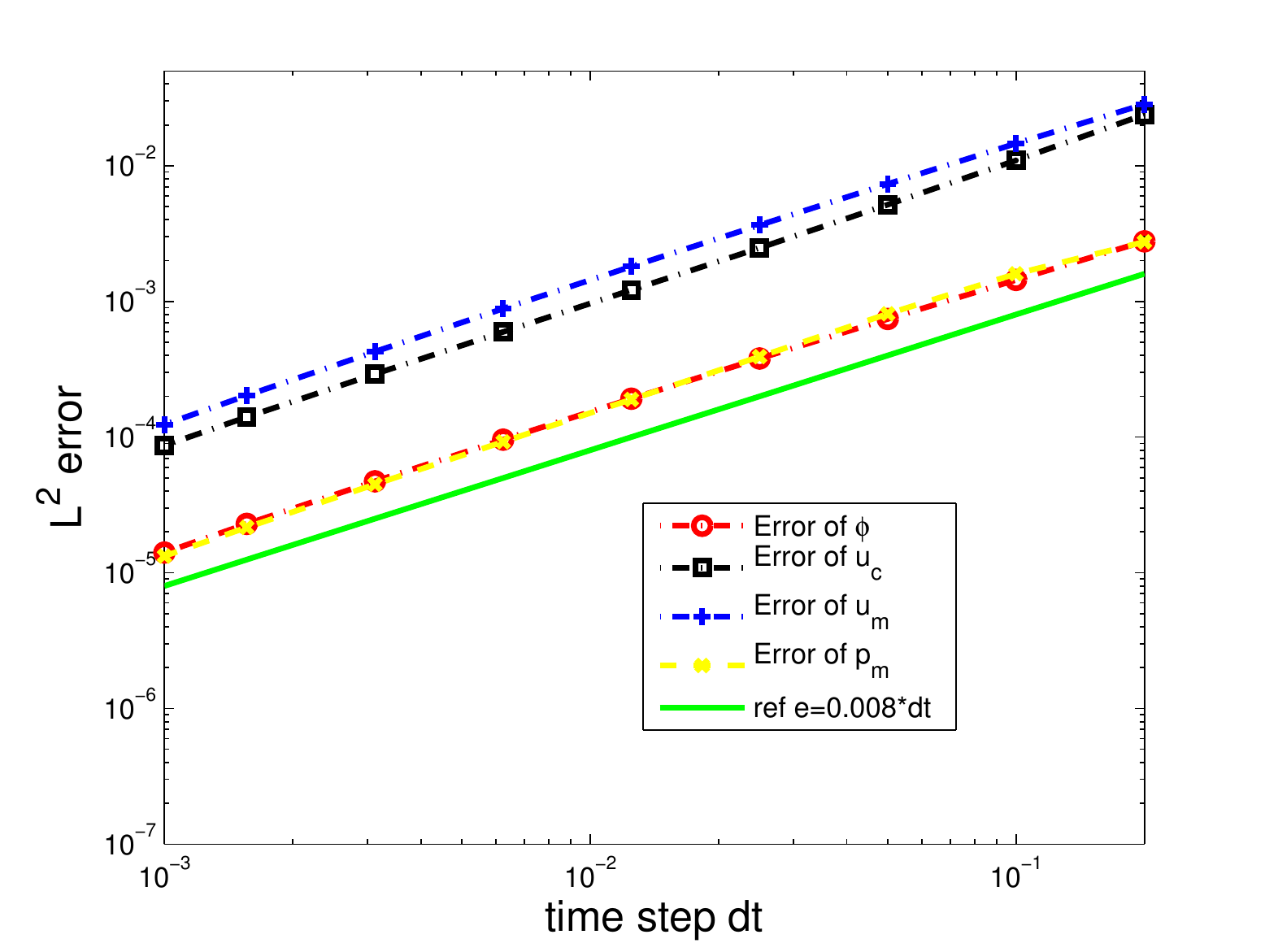}  
        \caption{Log-Log plot of the error measured in $L^2$ norm as a function of time step $\tau$ for $\mathbf{u}_c$, $\mathbf{u}_m$, $p_m$, and $\phi$. The solid green line is the reference line $e=0.008 \tau$. The final time is $T=1.0$. $h=0.01$. The other parameters are set to be unity.}
 \label{ConvTest}
\end{figure}

Next, we demonstrate numerically that our schemes satisfy discrete energy laws, i.e., the discrete energy $\mathcal{E}^{k}$ defined in \eqref{DisEner} is nonincreasing in time. We perform the classical numerical experiemnt of spinodal decomposition and coarsening. The initial velocities are the same as in the convergence test.  For the initial condition of the order parameter, we take a random field of values $\varphi_0=\bar{\phi}+r(x,y)$ with an average composition $\bar{\phi}=-0.05$ and random $r \in [-0.05, 0.05]$. The parameters in this experiment are taken to be: $\frac{\rho_0}{\chi}=0.01, \epsilon=0.01, \nu_c=\nu_m=0.1, \Pi=1, \gamma=0.1, M=0.1$. The evolution of the discrete energy $\mathcal{E}^{k}$ is shown in Fig. \ref{EnergyEvo} where $h=0.01, \tau=0.1$.

\begin{figure}[h!]
        \centering
                \includegraphics[width=0.9\textwidth]{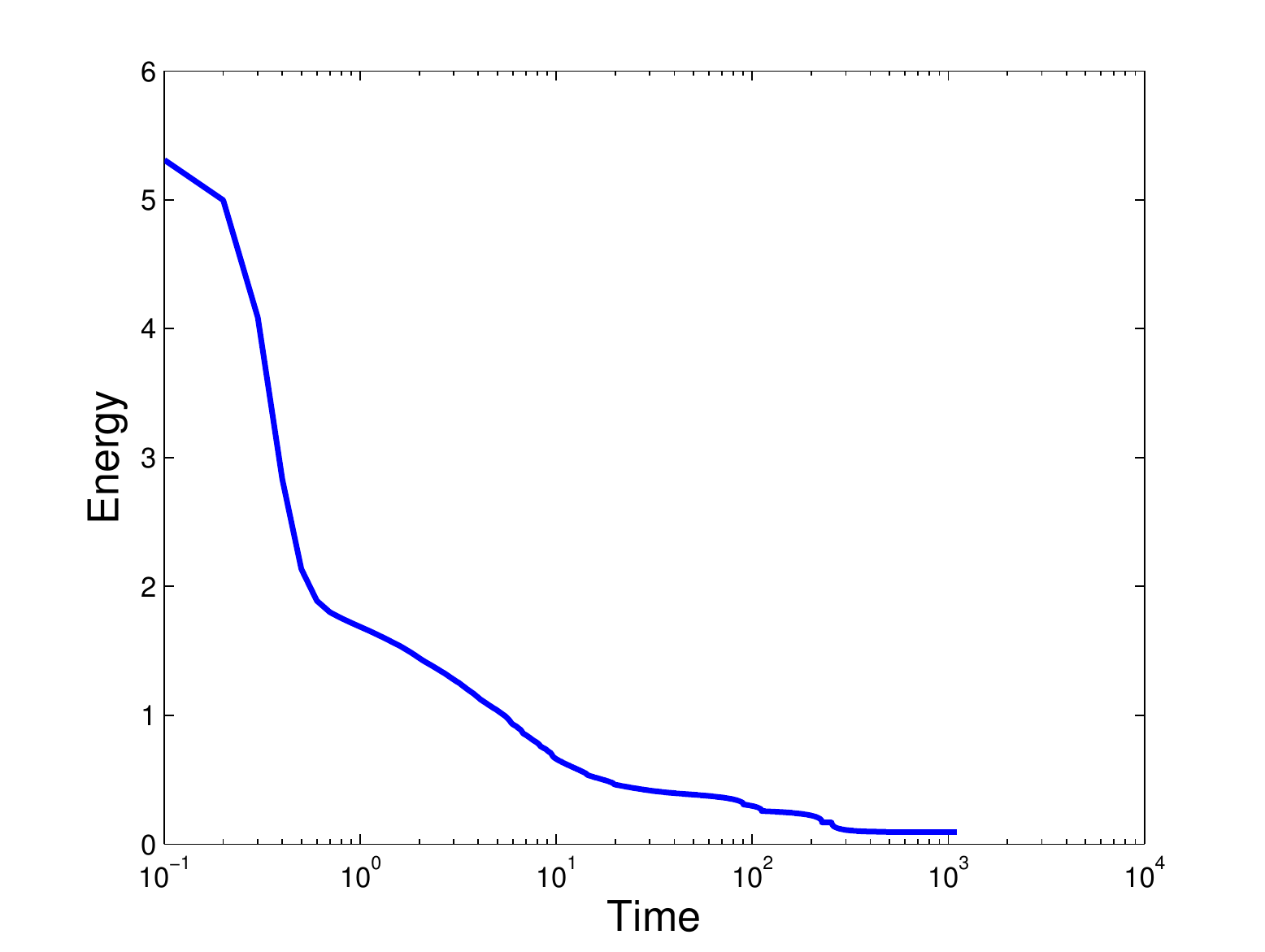} 
        \caption{Time evolution of the discrete energy $\mathcal{E}^{k}$ in the simulation of spinodal decomposition; $\tau= 0.1$.}     
        \label{EnergyEvo}  
\end{figure}

\subsection{Boundary-driven flow}
In this example, we consider a horizontal channel $\Omega=[0,2]\times [0,1]$ with the domain interface boundary $\{1\}\times[0,1]$ separating the conduit $\Omega_c=[0,1] \times [0,1]$ and the porous media $\Omega_m=[1,2] \times [0,1]$. 

The set-up of the experiment is as follows. We impose an inflow boundary condition of parabolic profile on part of the left boundary $\Gamma_{in}:=\{0\} \times [0.4, 0.6]$, i.e., $u1_c=-100a(y-0.4)(y-0.6)$ on $\Gamma_{in}$. On the right boundary $\Gamma_{out}:=\{2\} \times [0,1]$, ambient pressure is prescribed for the Darcy pressure, i.e.,  $P_m=0$ on $\Gamma_{out}$. The rest of the boundary conditions are the same as given in \eqref{IBC0}-\eqref{IBC1}. The initial condition $\ub_c|_{t=0}$ is given as the solution of the steady-state Stokes equation with the same injection boundary condition in the whole domain. An initial Darcy velocity is then  determined by using the explicit interface boundary conditions derived from Stokes fluid fields. The initial order parameter is set to be $\phi_0= -\tanh\big((0.15-\sqrt{(x-0.4)^2+(y-0.5)^2})/\sqrt{2.0\epsilon}\big)$. The initial order parameter and initial horizontal velocity are shown in Fig. \ref{IniOPVelB}.

\begin{figure}[h]
\centering
\begin{overpic}[width=0.5\textwidth]{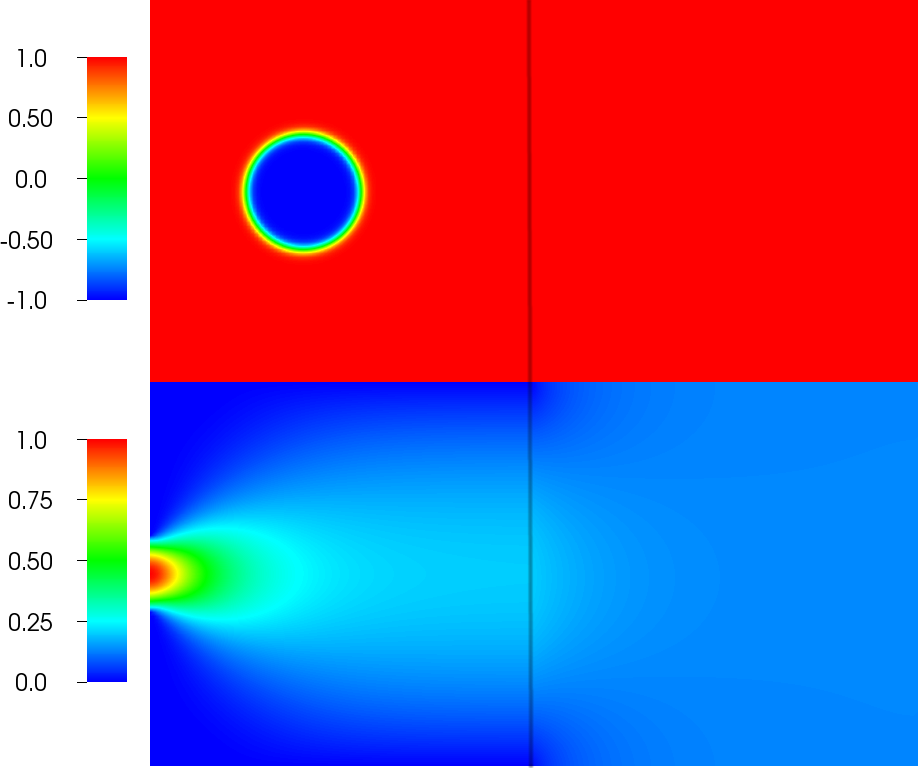}
    \put (0,80) {(a) $\phi_0$}
    \put (0,40) {(b) $u1_0$}
  \end{overpic}
  \caption{Filled contour plots of (a) the initial order parameter $\phi_0$,  and (b) the initial horizontal velocity  $u1_0$. The domain interface is at $\{1\}\times[0,1]$ separating the conduit (left) and the porous media (right). }
  \label{IniOPVelB}
\end{figure}

The parameters in this simulation are listed as $\epsilon=0.01$, $\gamma=0.001$, $\Pi=0.001$, $\nu(\varphi_c)=\nu(\varphi_m)=0.1$, $\alpha_{BJSJ}=0.1$, $a=1.0$, $M(\varphi)=\epsilon\sqrt{(1-\varphi^2)^2+\epsilon^2}$. We remark that the modified degenerate mobility function $M(\varphi)$ limits the chemical diffusion in the diffuse interface region. In the bulk of  each fluid region, the mobility is essentially $\epsilon^2$. We employ three meshes for the computation of Cahn-Hilliard equation, Darcy equation and Stokes equation, respectively, thanks in part to the complete decoupling of the three equations. The temporal time step-size is $\tau=0.001$ for accuracy.

Fig \ref{BounDr}  shows some snapshots of the droplet passing through the domain interface under the influence of boundary-driven flow. We note that the surface tension parameter is relatively small ($\gamma=0.001$) compared to the maximum velocity on the inflow boundary $u1(0, 0.5)=1$. The round droplet quickly deforms into a cap shape with the flat side facing the injection boundary (a). A dimple is formed in (b), as the fluid velocity takes the maximum value at the center line. As it moves through the domain interface at $\{1\} \times [0,1]$, the front of the droplet (with respect to fluid flow) becomes flatter, and elongates in the vertical direction, cf. (c) and (d). This is due to conservation of mass and the fact that  the magnitude of the velocity in porous media is significantly smaller than that in conduit. Once the droplet enters the porous media, the shape remains comparatively steady. One can see that the upper and lower ``tip'' of the droplet becomes soft (compare (e) to (f)) as a result of the surface tension effect.

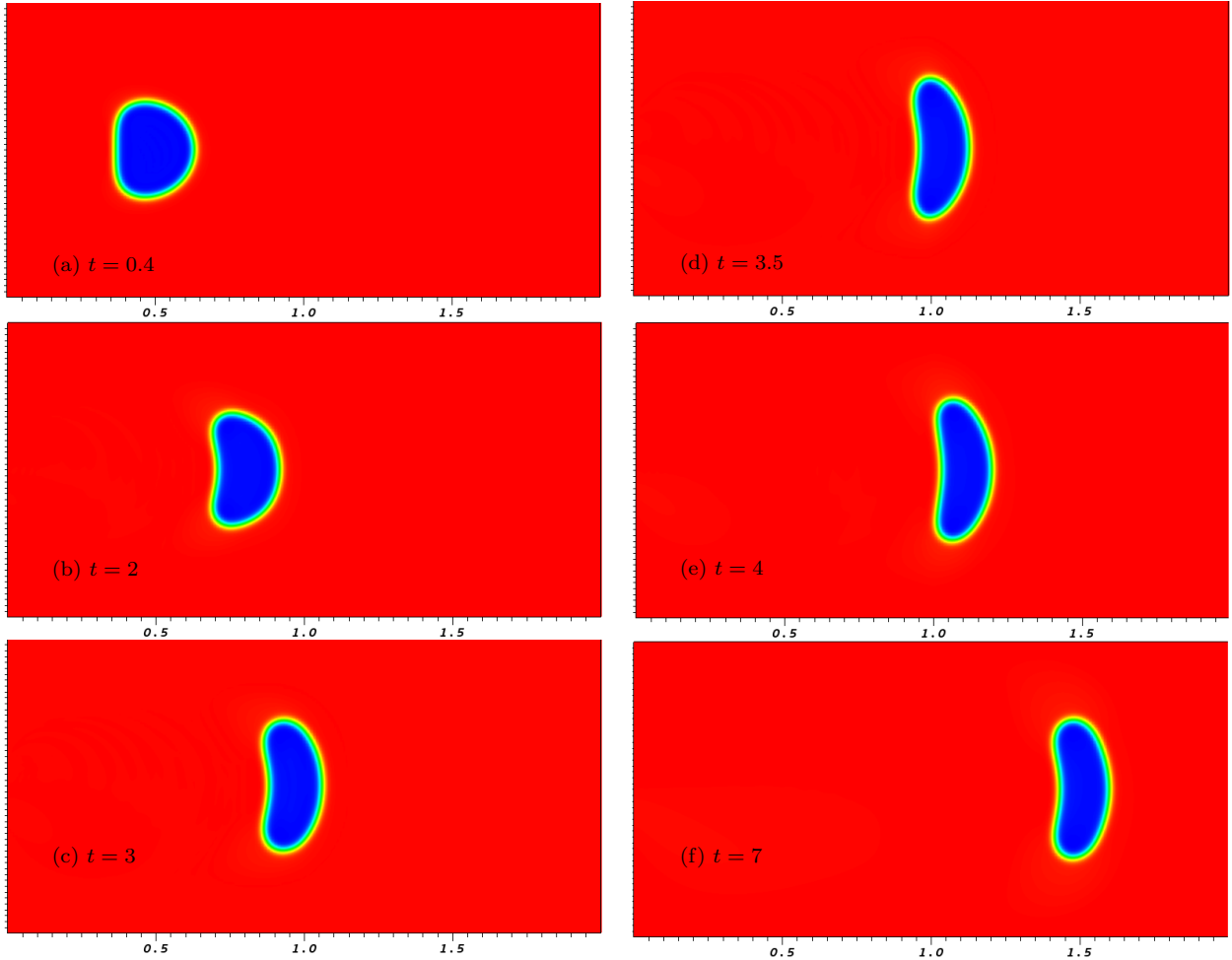
\begin{figure}[h!]
\centering
\begin{tabular}{cc}
\begin{overpic}[width=0.5\textwidth]{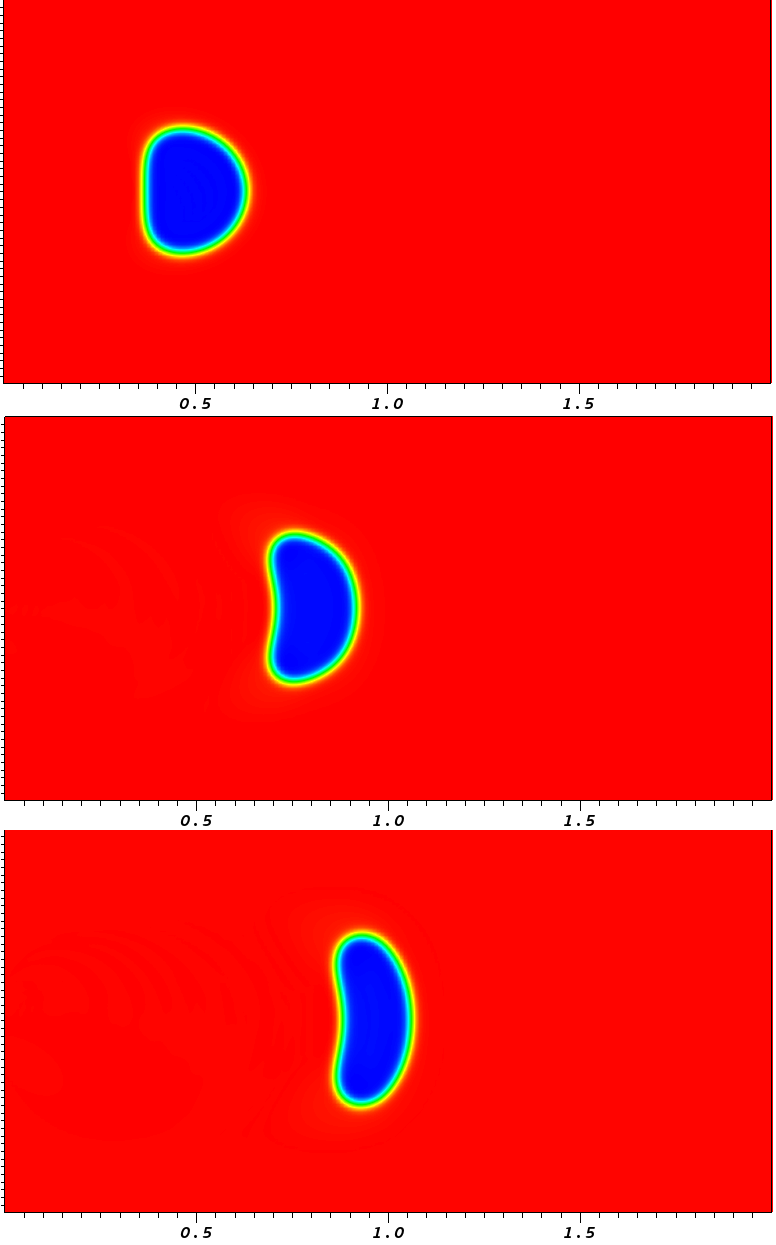}
    \put (5,72) {(a) $t=0.4$}
    \put (5,40) {(b) $t=2$}
    \put (5,10) {(c) $t=3$}
  \end{overpic}

  &
  \begin{overpic}[width=0.5\textwidth]{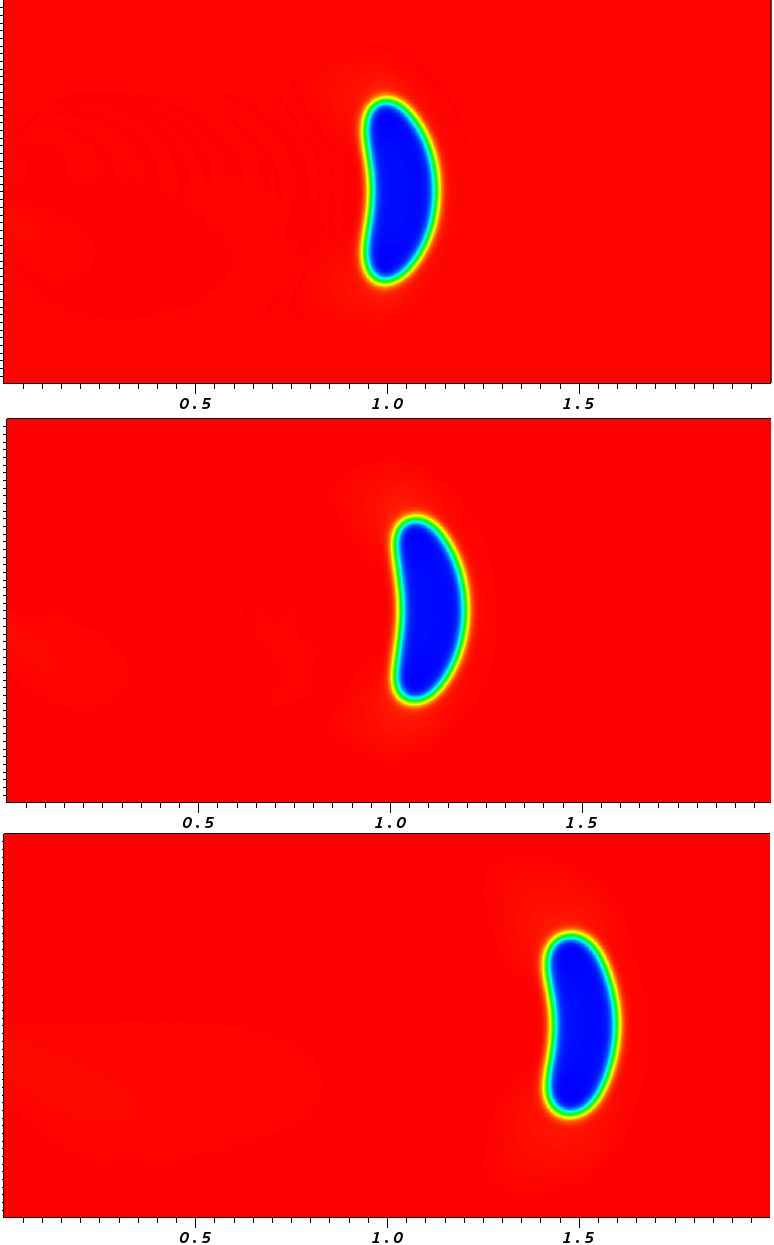}
    \put (5,72) {(d) $t=3.5$}
    \put (5,40) {(e) $t=4$}
    \put (5,10) {(f) $t=7$}
  \end{overpic}

\end{tabular}
\caption{Evolution of a droplet driven by boundary injection in a karstic domain. From top to bottom, first column, (a) $t=0.4$, (b) $t=2$, (c) $t=3$; second column, (d) $t=3.5$, (e) $t=4$, (f) $t=7$. Blue color $\phi \approx -1$  and red color $\phi \approx 1$.}

\label{BounDr}

\end{figure}

\subsection{Buoyancy-driven flow}
Here, the karst geometry is modelled by a long tube $\Omega=[0,1]\times[-1,1]$ with the conduit $\Omega_c=[0,1]\times [-1,0]$ and porous media $\Omega_m=[0,1] \times [0,1]$. The interface boundary is at $[0,1] \times \{0\}$.

In this experiment, we consider a binary system  where the densities of the two fluids are different. But the density difference is small so that a Boussinesq approximation is applicable. Specifically, a buoyancy term $G(\rho(\phi)-\bar{\rho})\hat{y}:=B(\phi-\bar{\phi})\hat{y}$ is added to Stokes equation \eqref{HSCH-NSCH1} and Darcy equation \eqref{HSCH-NSCH4}, respectively. Here $\rho(\phi)=\frac{1+\phi}{2}\rho_1+\frac{1-\phi}{2}\rho_2, $ $\bar{\rho}$ and $\bar{\phi}$ are the spatial averages of $\rho$ and $\phi$, $B=G\frac{\rho_1-\rho_2}{2}$, $\hat{y}=(0,1)^{T}$. We consider  a lighter round bubble released in an initially quiescent  heavier fluids. The boundary conditions are given in  \eqref{IBC0}-\eqref{IBC1}. Most of the parameters used in this simulation are the same as those in boundary driven flow case, except $B=2.0$ and $\Pi=0.01$.

The filled contour plots in gray scale of the rising bubble are shown in Fig. \ref{BuoyDr}. As the bubble rises in the conduit domain, it deforms into an ellipsoid. When it passes through the domain interface, one can clearly see an interface separating the bubble in conduit and in porous medium. Two corners of the bubble in some sense are formed along the domain interface with the part in the conduit being wider than that in porous medium. A tail of the bubble is seen later as it leaves the domain interface. The tail is eventually smoothed out by the surface tension effect.

\begin{figure}[h!]
\centering
\begin{tabular}{c}
\begin{overpic}[width=0.95\textwidth]{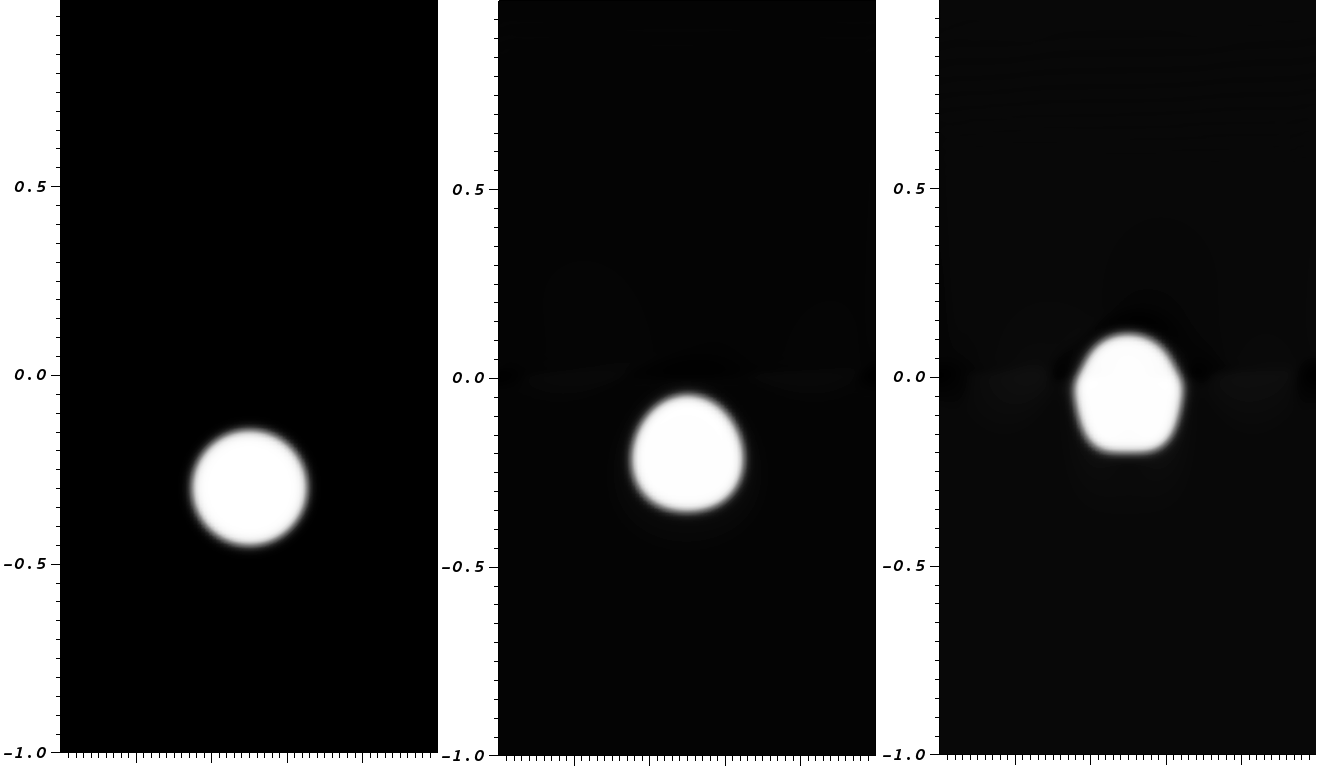}
    \put (6,58.5) {(a) $t=0$}
    \put (38,58.5) {(b) $t=0.75$}
    \put (72,58.5) {(c) $t=1.7$}
  \end{overpic}
  
  \\
  \begin{overpic}[width=0.95\textwidth]{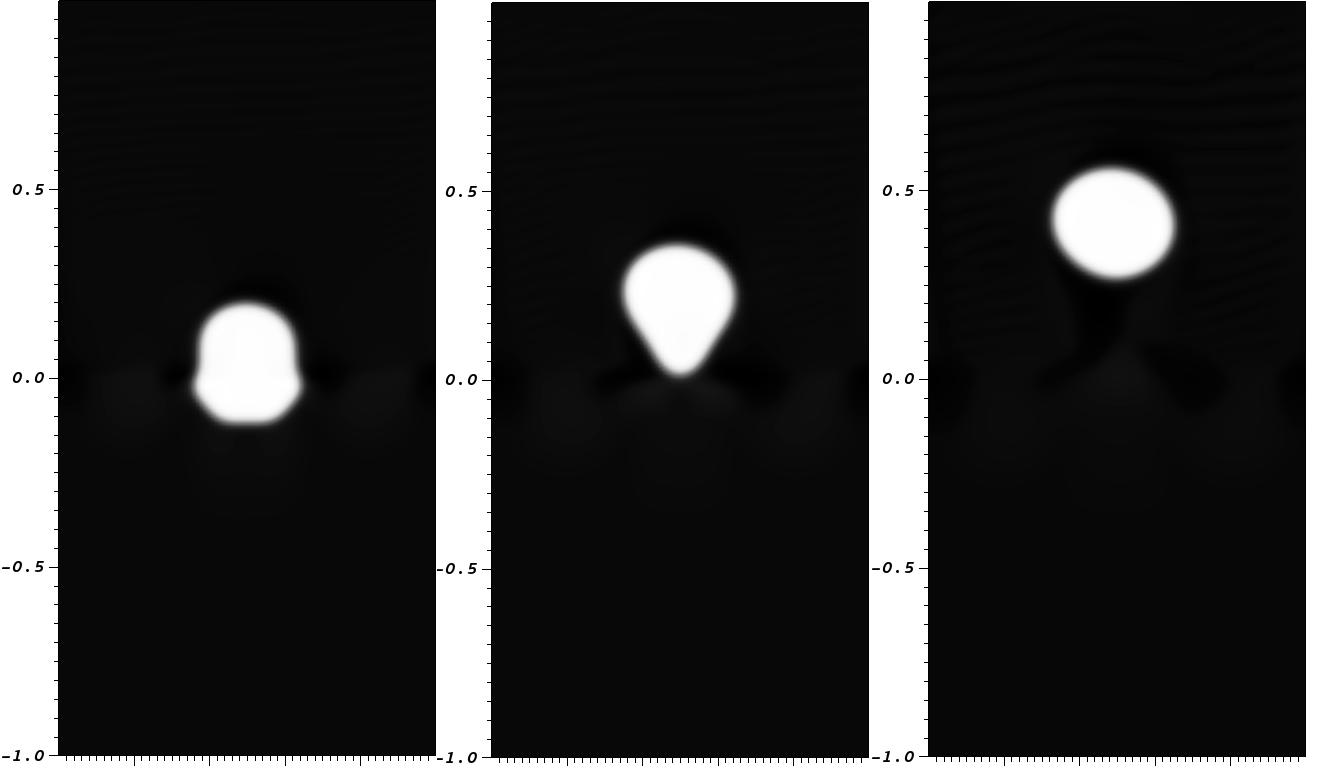}
    \put (6,-2.5) {(d) $t=1.97$}
    \put (38,-2.5) {(e) $t=3.2$}
    \put (72,-2.5) {(f) $t=4.5$}
  \end{overpic}
  
\end{tabular}
\caption{Snapshots of bubble rising due to buoyancy in a karstic domain. From left to right, first row, (a) $t=0.0$, (b) $t=0.75$, (c) $t=1.7$; second row, (d) $t=1.97$, (e) $t=3.2$, (f) $t=4.5$. White color $\phi \approx -1$  and black color $\phi \approx 1$. The domain interface is at $[0,1] \times \{0\}$}

\label{BuoyDr}

\end{figure}
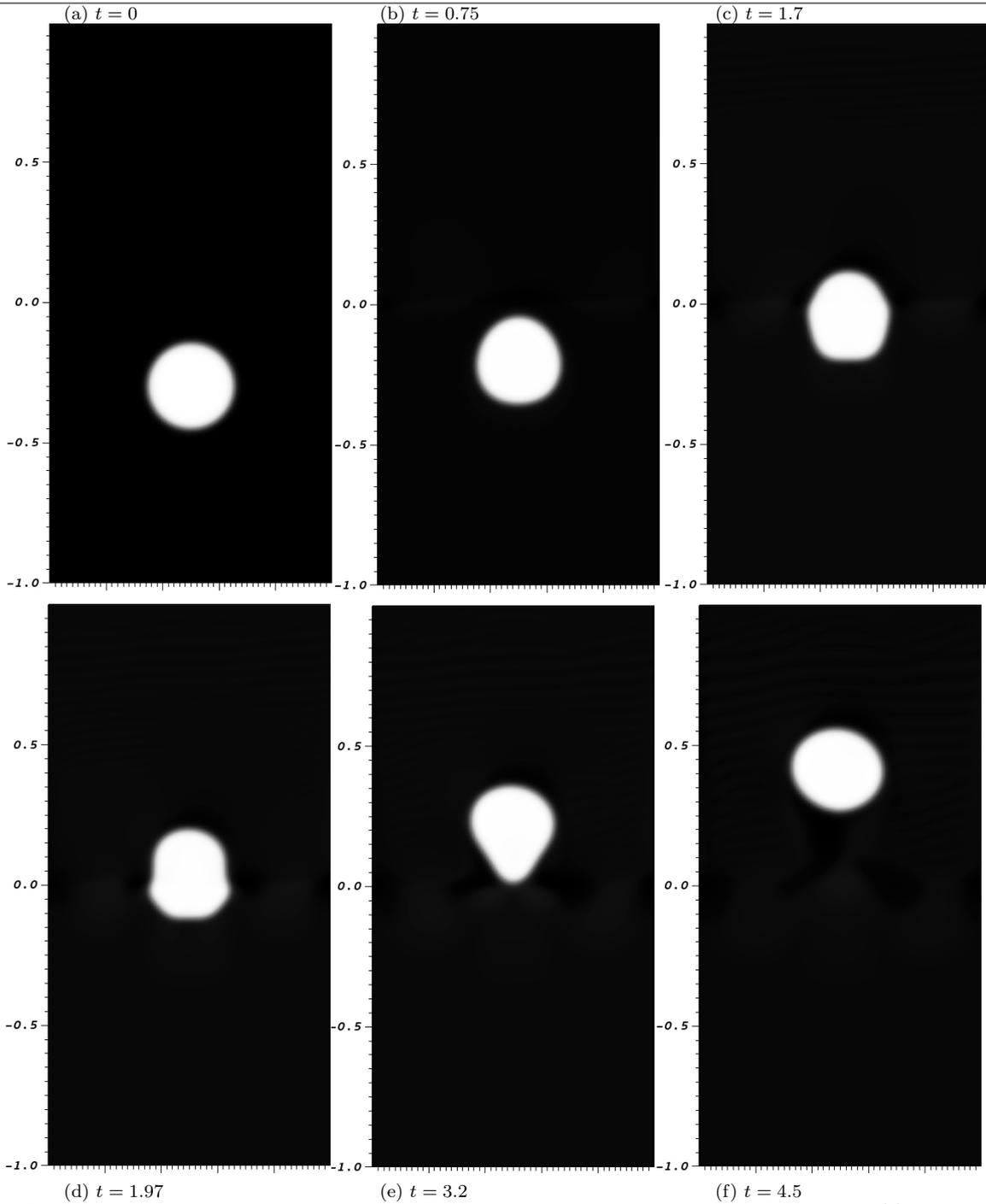

\section{Conclusions}
We have proposed, analyzed and implemented two novel uniquely solvable, energy stable decoupled algorithms for the Cahn-Hilliard-Stokes-Darcy system which models two-phase flows in karstic geometry. The decoupling of the phase-field and the velocity field is realized via an intermediate velocity that takes into account the capillary force term only. Therefore, we are required to solve a strictly convex variational problem for the phase field part at each time step. The phase-field update is independent of the velocity update for both schemes (and hence decoupled). For the first scheme, the velocity field is governed by the linear Stokes-Darcy system once the phase-field is updated. For the second scheme, we further decouple the linear Stokes-Darcy system into a linear Darcy type equation and a Stokes type system. Therefore, appropriate legacy code for the Cahn-Hilliard equations, the Stokes-Darcy system, the Darcy equation and the Stokes system can be utilized.
  We have also established the unique solvability and energy stability of both algorithms rigorously. 
  So far as we know, these two schemes are the first set of decoupled uniquely solvable and energy stable algorithms for simulating two phase flows in karstic geometry. 
  Two physically interesting numerical experiments are conducted, one buoyancy driven and one boundary driven. The numerics illustrate the efficiency and the stability of the schemes.
  
  The error estimates of the schemes proposed here will be the subject of a future work.


\bibliographystyle{spmpsci}      
\bibliography{CHSD_1011.bib}

\end{document}